\documentclass[11pt]{article}
\usepackage[mathcal,mathbf]{euler}
\usepackage{inputenc, amsmath,enumerate,fancyhdr,amssymb,amsfonts,accents, url, color, float,graphicx, subfigure, grffile, float, datetime}
\restylefloat{figure}

\usepackage{MyCommands}

\usepackage{MinConeCommands}

\newcommand{\bibA}{refs_minBall}
\newcommand{\bibB}{refs_QP}	
\newcommand{\bibC}{refs_minCone}	
\newcommand{\bibD}{refs_various}	
\newcommand{\bibE}{refs_SOCP}

\title{A dual Simplex-type algorithm for the smallest enclosing ball of balls and related problems}

\author{Marta Cavaleiro\thanks{
		Rutgers University, MSIS Department \& RUTCOR, 
		100 Rockafellar Rd, Piscataway, NJ 08854.
		 \texttt{marta.cavaleiro@rutgers.edu} }  
	\and Farid Alizadeh\thanks{
		Rutgers University, MSIS Department \& RUTCOR,
		100 Rockafellar Rd, Piscataway, NJ 08854. 		
	 \texttt{farid.alizadeh@rutgers.edu} }
	}

\newdate{date}{04}{12}{2018}
\date{\displaydate{date}}

\begin{document}
	
\maketitle

\abstract{We define the notion of infimum of a set of points with respect to the second order cone. This problem can be showed to be equivalent to the minimum ball containing a set of balls problem and to the maximum intersecting ball problem, as well as others. We present a dual algorithm which can be viewed as an extension of the simplex method to solve this problem. }


\section{Introduction} Let $\Q$ denote the second-order cone, $\Q:=\left \{x:=(x_0; \xb)\in\RR^{n}:\,\norm{\xb}_2\leq x_0\right\}$. Given a set of points $\Pp=\{p_1,...,p_m\}\subset\RR^{n}$, we consider the problem defined as follows
\begin{equation} \begin{array}{rl}\label{primal}
Inf_\Q(\Pp) := \displaystyle\max_{x} \,& \inner{e_1}{x}  \\
\displaystyle 	  \st & x\leqQ p_i, \, i=1,...,m\tag{P}
\end{array}\end{equation}		
\noindent with $e_1=(1,0,...,0)$. We define this problem as the \emph{infimum of $\Pp$ with respect to~$\Q$} (given its resemblance with the problem of finding the minimum number of a set of numbers with a linear program). Throughout the article we shall refer to this problem as $Inf_\Q(\Pp)$, or simply (\ref{primal}).

Denote by $x^*$ the optimal solution to (\ref{primal}). One possible geometric interpretation of the problem in question is to find $x^*$ as ``high'' as possible (when height is defined as the value of $x_0$)  such that $x^*+\Q$ covers all points of set $\Pp$, in the sense that they either fall inside or on the boundary of $x^*+\Q$ (Figure \ref{fig:infCone}). But more interestingly, it is possible to prove that (\ref{primal}) is in fact equivalent to several relevant problems in computational geometry, such as the smallest enclosing ball of balls or the largest intersecting ball.

\begin{center}
	\begin{figure}[h]
		\centering
		\includegraphics[width=4.5cm]{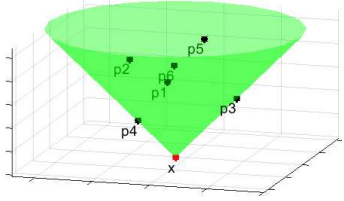}
		\caption{Geometric interpretation of the \emph{infimum of with respect to~$\Q$}.}
		\label{fig:infCone}
	\end{figure}
\end{center}

Problem (\ref{primal}) is a Second-Order Cone Program (SOCP) and so it can be solved in polynomial time using interior point methods, such as the primal-dual path following algorithms \cite{Nesterov98, Nesterov97}.%the number of iterations is $\bigO(\sqrt{m}\log{\frac1\epsilon})$}
These algorithms generate interior-points for the primal and dual problems that follow the so-called \emph{central path}, which converges to a primal-dual optimal solution in the limit. Each iteration, in terms of computational work, basically consists on the solution of a linear system to compute the search direction resulting from the application of Newton's method to the KKT conditions of a modified (\ref{primal}) with a \emph{logarithmic barrier function}. The matrix in question is positive definite (see e.g. \cite{Alizadeh03}), and usually its Cholesky factorization is used. In terms of our problem, the matrix has size $n\times n$ (note that the standard form cone LP is the dual), and would take $\bigO(mn^2)$ basic arithmetic operations to be computed \cite{Zhou05} 
resulting in a iteration with $\bigO({mn^2}+n^3)$ computational complexity. Therefore, computing a solution with an error $\epsilon$ with an interior-point method would have an overall complexity of $\bigO((m+n)n^2\sqrt{m}\log{\frac1\epsilon}))$, though in practice it has been observed that the number of iterations is often very small and independent on $m$. 
Interior point methods have been used to solve (\ref{primal}) in particular in \cite{Kumar03, Zhou05}. Specifically, in \cite{Zhou05} an algorithm that computes a $(1+\epsilon)$-approximation in $\bigO(\frac{mn}\epsilon+\frac1{\epsilon^{4.5}}\log\frac1\epsilon)$ time is presented.

\medskip

In this paper we introduce a dual algorithm for (\ref{primal}) that mechanically is analogous to the Simplex method for Linear Programming. Each (major) iteration of our algorithm starts with a primal solution corresponding to a \emph{dual feasible basic solution}, but then, instead of a line search, our algorithm will perform a sequence of \emph{exact curve searches} until it arrives to a new \emph{dual feasible basic solution} with a better objective function value. We will define the notion of \emph{basis} as well as \emph{dual feasible basic solution} applied to our problem, which we will rename as \emph{support set} and \emph{dual feasible S-pair}, respectively. Mechanically speaking, our algorithm shares similarities with the dual active set algorithm for strictly convex QPs by Goldfarb and Idnani in \cite{Goldfarb83}, and with the work of Dearing and Zeck who proposed a dual simplex method for the minimum enclosing ball of points in \cite{Dearing09}.

\smallskip

Besides providing an exact solution, one advantage that a dual simplex algorithm has over interior point methods is that, after solving the problem, if it suffers small changes (e.g. adding an extra constraint), the dual simplex method will usually require a small number of iterations to calculate the new solution when it starts with the original solution. Another advantage of simplex-type algorithms in general is that they generate \emph{basic} solutions. For instance in the case of the minimum enclosing ball of balls (and in particular of points too), that will be able to tell us which input balls (points) determine the smallest enclosing ball. On the other hand, we do not have any overall polynomial complexity guarantees for the algorithm.%the algorithm can take an exponential number of iterations.

\medskip

The paper is organized as follows. In section \ref{sec:geometric} we show that problem (\ref{primal}) is equivalent to relevant problems in computational geometry involving hyperspheres. Section \ref{sec:prelims} presents important theoretical background to the algorithm, such as duality results and the definition of \emph{support set}. The dual simplex algorithm for (\ref{primal}) is then introduced in section \ref{sec:algorithm}. We then briefly explain the implementation details in section \ref{sec:implementation} and show some computational results in section \ref{sec:results}.

\section{Equivalent geometric problems}\label{sec:geometric} We now show that problem (\ref{primal}) is equivalent to several relevant problems in computational geometry such as the smallest enclosing ball of balls. For what follows, denote by $B(c, r)$ a ball with the Euclidean norm with center at $c\in \RR^{n-1}$ and radius $r\geq 0$. 

\smallskip\noindent\textbf{The smallest enclosing ball problem.} The classical and well studied problem of enclosing a set of Euclidean balls with an Euclidean ball of smallest radius can be reduced to (\ref{primal}) as a consequence of the fact that $B(c_1, r_1)\cap  B(c_2, r_2)\neq\emptyset$ if and only if $\norm{c_2-c_1}\leq r_2+r_1$. Thus, the problem of enclosing a set $\B$ of balls $B(c_i, r_i)$, $i=1,...,m$ with a ball $B(c,r)$ of minimum radius can be solved by (\ref{primal}) considering $\Pp=\{(-r_i; c_i),\, i=1,...,m\}$. The optimal ball will then be given by the center $c=\xb^*$ and radius $r=-x_0^*$. When $r_i=0$ for all $i$, the problem reduces to the smallest enclosing ball of points. Figure \ref{fig:MinEB} illustrates this equivalence: the minimum enclosing ball of set $\B$ is the intersection of the cone $x^*+\Q$ with the plane $x_{0}=0$. 

The smallest enclosing ball of balls, and in particular of points, is a classical problem in computational geometry. %reported to go back to the $19^{th}$ century \cite{Sylvester}.
This problem has been extensively studied, specially from a combinatorial point of view, in particular in the \emph{LP-type} framework, see e.g. \cite{Dyer92, Fischer04, Matousek96,Megiddo89, Welzl91}. In particular, the minimum enclosing ball of points can easily be converted in a Quadratic Program (QP) and solved using off-the-shelf QP solvers. G{\"a}rtner and Sch{\"o}nherr \cite{Gartner00} developed a generalization of the simplex method for QP with the goal of targeting geometric QPs, with one of the main applications being the MB problem, while later, Fischer and G{\"a}rtner \cite{Fischer04} proposed an algorithm with a pivoting scheme resembling the simplex method for LP based on previous ideas from \cite{Hopp96}. Using related ideas, Dearing and Zeck \cite{Dearing09} developed a dual algorithm for the MB problem. This algorithm was further improved in \cite{Cavaleiro18}. Several approximation algorithms have also been developed focusing on finding an $\epsilon$-\emph{core set}, \cite{Badoiu02}, that is a subset of $\Ss\subset\Pp$ that has the property that the smallest ball containing $\Ss$ once expanded by $1+\epsilon$ covers $\Pp$. A surprising fact is the existence of an $\epsilon$-core set of size at most $\lceil\frac1\epsilon\rceil$, independent of the dimension $n$, for any point set $\Pp\subset\RR^n$, \cite{Kumar03, Badoiu03}. Several algorithms focused on finding $\epsilon$-core sets have been proposed \cite{Badoiu03,Badoiu02,Kumar03,Larsson13, Nielsen09,Yildirim08}.

\begin{figure}[h]
	\centering
	\subfigure[\emph{Cone view}]{\includegraphics[height=3.5cm]{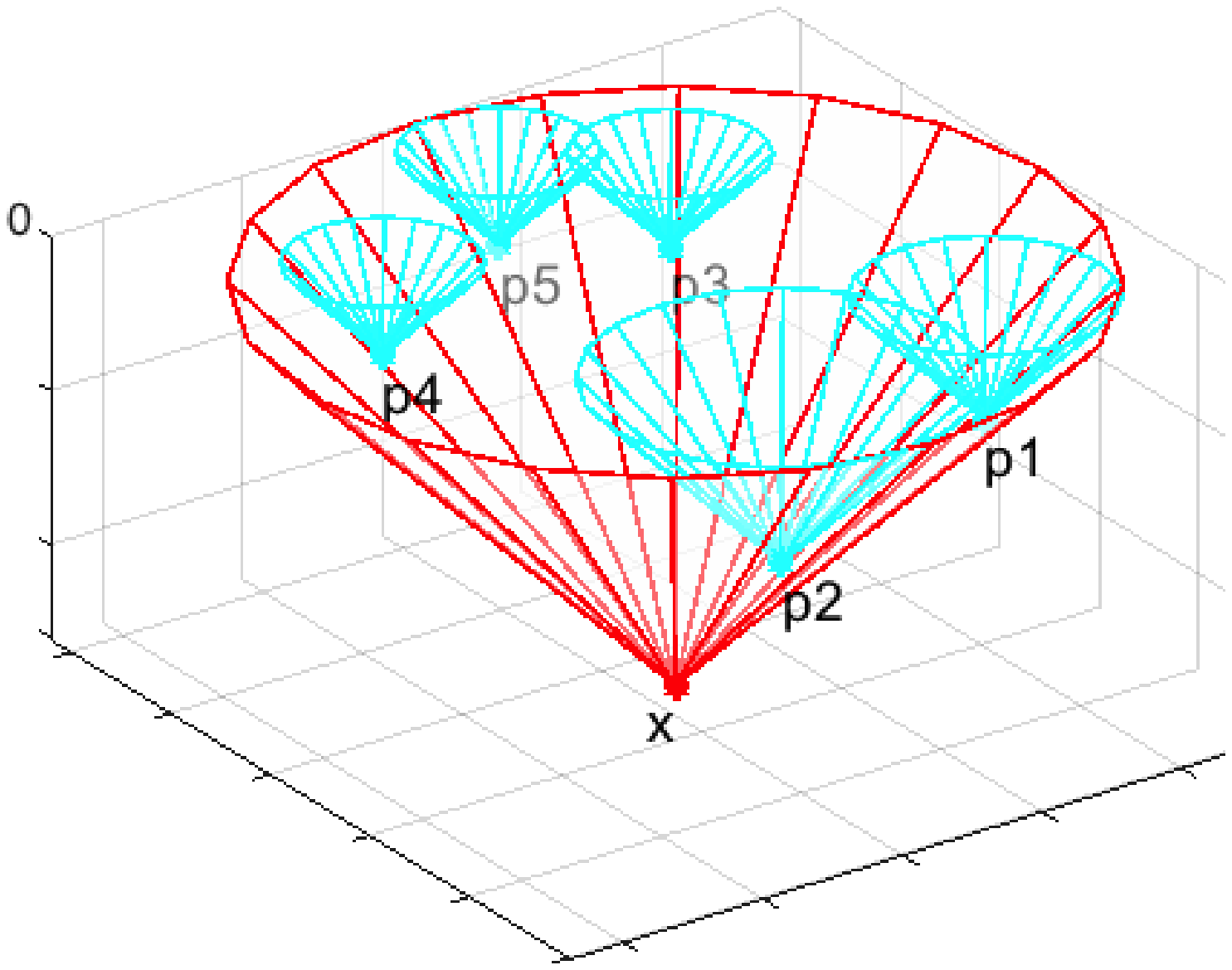}}
	\hspace{2cm}
	\subfigure[\emph{View from above at the cross-section at $x_{0}=0$}]{\includegraphics[height=3.7cm]{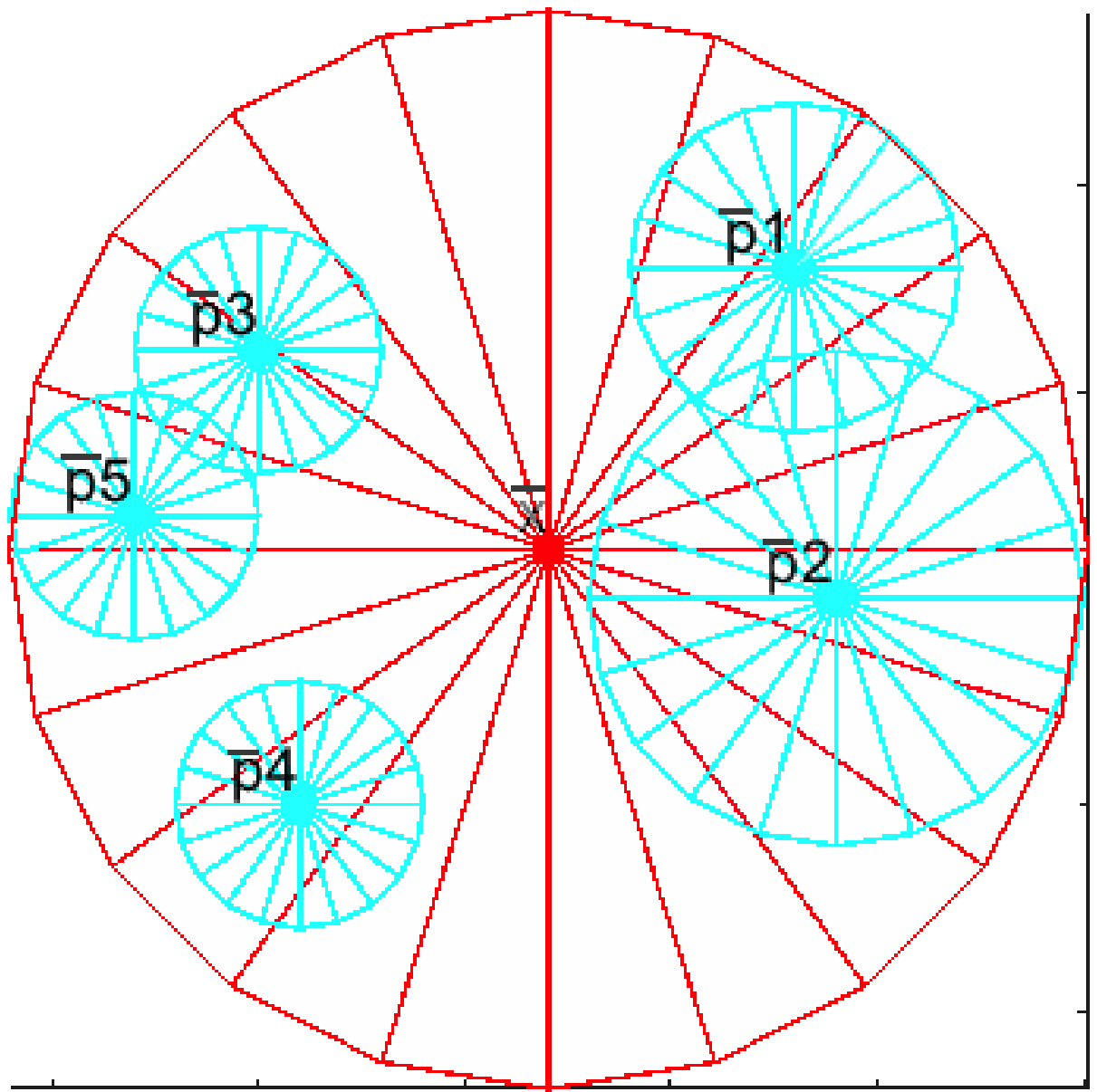}}
	\vspace{-0.25cm}\caption{The smallest enclosing ball of a set of balls.}\label{fig:MinEB}
\end{figure}

\noindent\textbf{The smallest intersecting ball and the largest enclosed ball problems.} Consider now the \emph{smallest intersecting ball problem} of finding the Euclidean ball with smallest radius that intersects all balls $B(c_i, r_i)$, $i=1,...,m$. This problem also reduces to (\ref{primal}) given that $B(c_1, r_1)\cap  B(c_2, r_2)\neq\emptyset$ if and only if $\norm{c_2-c_1}\leq r_2+r_1$. Thus, considering $\Pp=\{(r_i; c_i),\, i=1,...,m\}$ in (\ref{primal}), the smallest ball that intersects all balls has center $c=\xb^*$ and radius $r=-x_0^*$. Figure \ref{fig:MinIB} illustrates the equivalence.
Balls $B(c_i, r_i)$, $i=1,...,m$, may all intersect, and so the smallest intersecting ball could be considered any point in the intersection. In such case, the smallest intersecting ball problem then looses its interest. In fact, when all balls intersect, the solution is such that $x_0^*>0$. This solution however is not meaningless, in fact, it is not difficult to see that, when $\bigcap B(c_i, r_i)\neq\emptyset$, ball $B(\xb,x_0^*)$ is the largest radius ball that is enclosed in the intersection of the balls. We shall refer to this problem as the \emph{largest enclosed ball problem} (see Figure \ref{fig:MaxEB}). For previous work on this problem see \cite{Morduk13, Nam12} and the references therein.
%Applications: \cite{Kalantari15}}

\begin{figure}[h]
	\centering
	\subfigure[\emph{Cone view}]{\includegraphics[height=3.5cm]{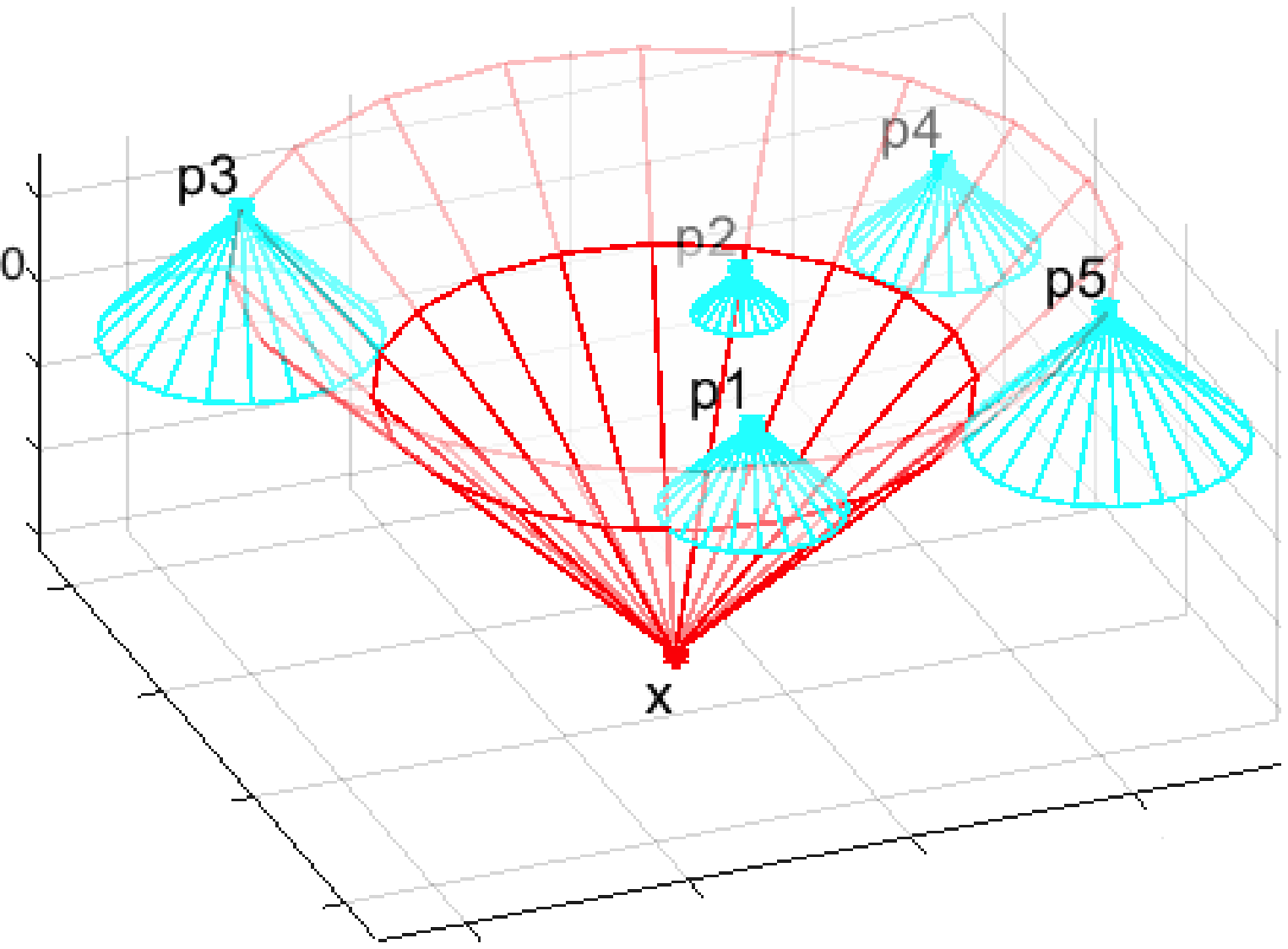}}
	\hspace{2cm}
	\subfigure[\emph{View from above at the cross-section at $x_{0}=0$}]{\includegraphics[height=3.5cm]{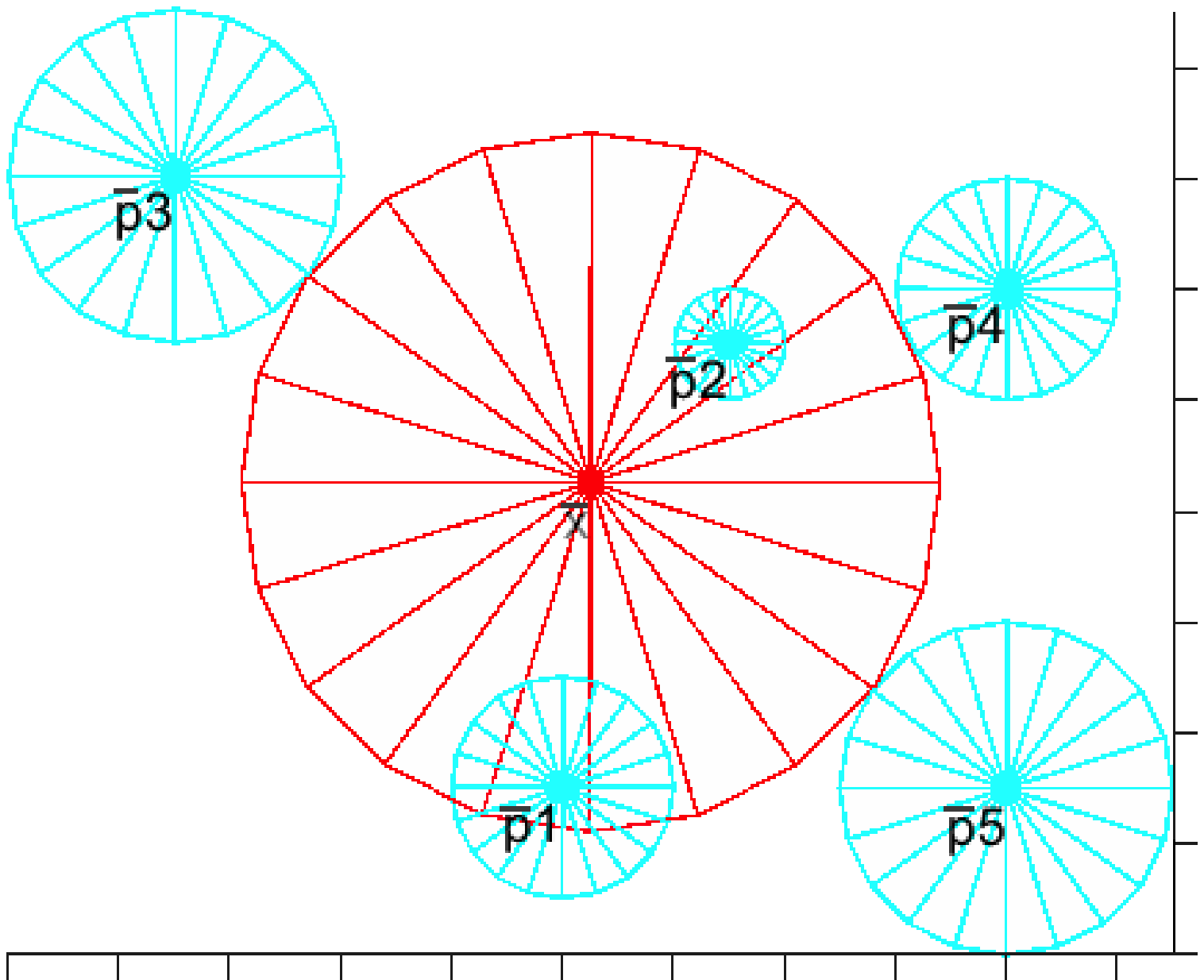}}
	\vspace{-0.25cm}
	\caption{The smallest intersecting ball of a set of balls.}\label{fig:MinIB}
\end{figure}

\begin{figure}[h]
	\centering
	\subfigure[\emph{Cone view}]{\includegraphics[height=3.5cm]{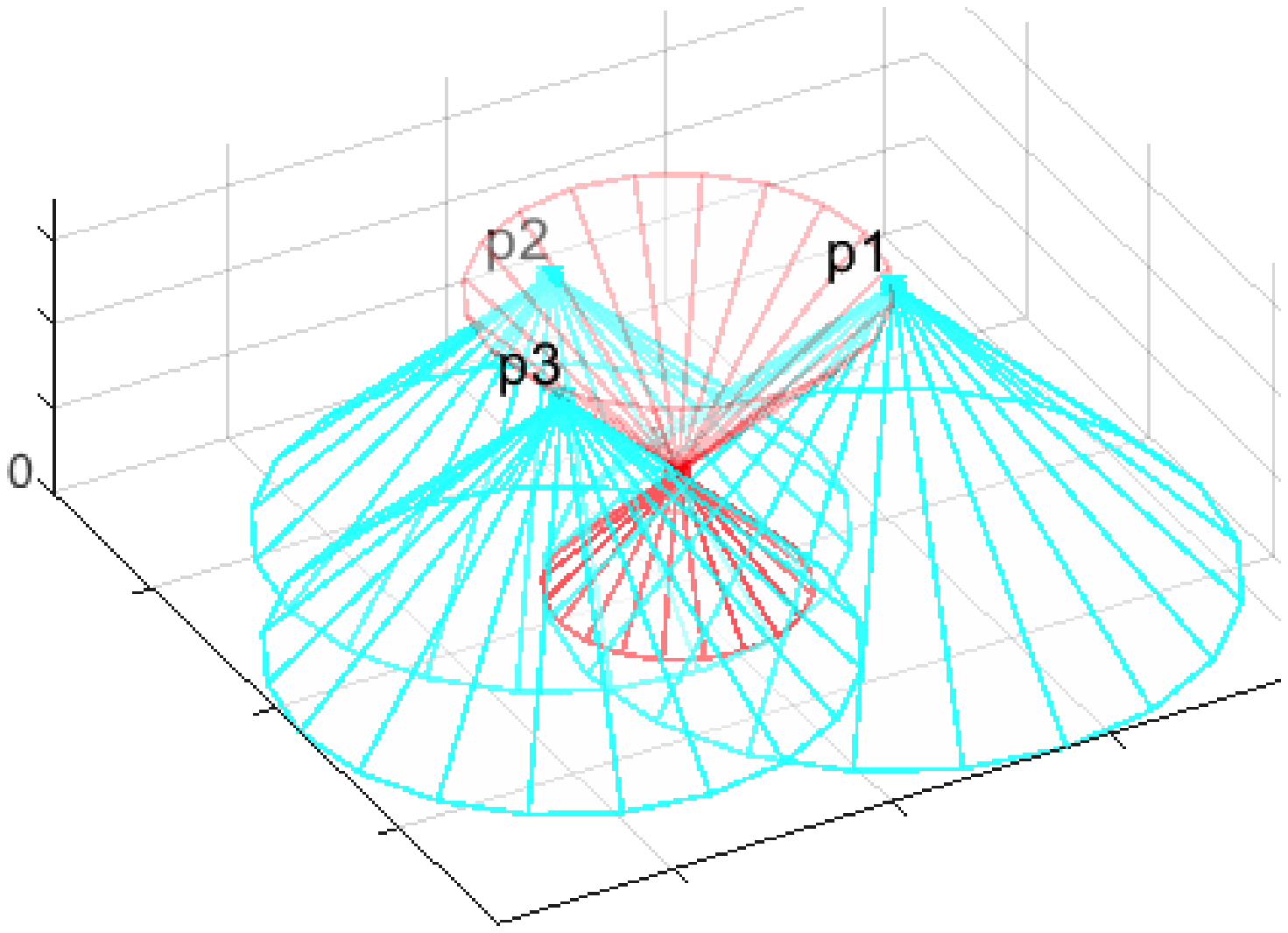}}
	\hspace{2cm}
	\subfigure[\emph{View from below at the cross-section at $x_{0}=0$}]{\includegraphics[height=3.5cm]{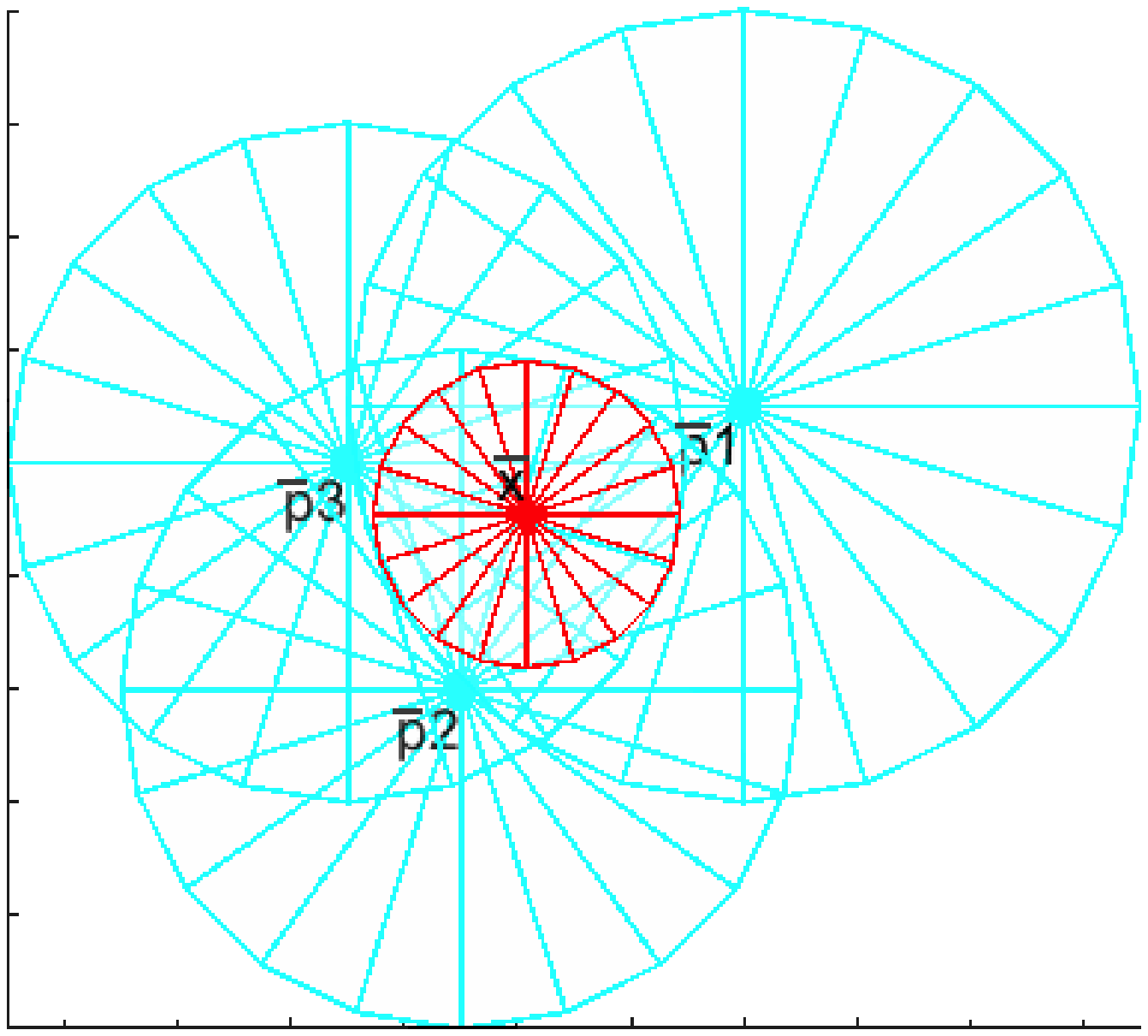}}
	\vspace{-0.25cm}\caption{The largest enclosed ball in a set of balls.}\label{fig:MaxEB}
\end{figure}
	
\noindent\textbf{The smallest intersecting and enclosed ball problem.} From the previous two sections, we conclude that the problem of finding a ball with smallest radius that simultaneously encloses balls $B(c_i, r_i)$, $i=1,...,m_1$, and intersects balls $B(c_j, r_j)$, $j=1,...,m_2$, can be solved by considering $\Pp = \{(-r_i; c_i),\, i=1,...,m_1\}\cup\{(r_j; c_j),\, j=1,...,m_2\}$. The optimal ball will then have center $\xb^*$ and radius $-x_0^*$. 
	%Can balls $B(c_j, r_j)$, $j=1,...,m_2$ all intersect each other? Yes, that's no problem!
	%Previous work on this problem include \cite{Nam14}.

\section{Preliminaries}\label{sec:prelims} Before we proceed, let us introduce/review some notation.
\begin{center}
	\begin{tabular}{l|p{12cm}}
		$\Pp$ & $\{p_1,...,p_m\}$ \\
		$p_i=(\pio;\pbi)$ & a point from $\Pp$\\
		$\overline{\Pp}$& set $\{\pb_1,...,\pb_m\}$   \\	
		$x=(x_0;\xb)$ & primal variables\\
		$y_i=(y_{i0};\yb_i)$ & dual variables, $i=1,...,m$
	\end{tabular}
\end{center}

%------------------------------------------------------------------------------------------------------------------------
\subsection{Duality and optimality conditions}

It is easy to see that the solution to (\ref{primal}) always exists and that it is unique. Moreover, it is possible to prove the following:

\begin{lemma}\label{lem:onePointSol}
	The solution to $\InfQP$ is $p_k\in\Pp$, for some $k=1,...,m$, if and only if $p_k\leqQ p_i$ for all $i=1,...,m$ and $p_{k0}=\min\{p_{i0},\, i=1,...,m\}$.
\end{lemma}

The dual problem of $\InfQP$ is
\begin{equation} \begin{array}{cl}\label{dual}
	\displaystyle\min_y & \displaystyle\sum_{i=1}^m \inner{p_i}{ y_i}\\
	 \st &\displaystyle\sum_{i=1}^m y_{i0} =1\\
	 &  \displaystyle\sum_{i=1}^m \yb_i = 0\\
	& y_i\geqQ 0, \quad i=1,..., m.\tag{D}
\end{array}\end{equation}	

\noindent The solution to (\ref{dual}) may not be unique.

\begin{lemma}\label{lem:strictFeas}
	Both primal problem (\ref{primal}) and dual problem (\ref{dual}) are strictly feasible; i.e. there exists a primal-feasible vector $x$ such that $x\lQ p_i$ for all $i=1,...,m$, and there exist dual-feasible $y_1,...,y_m$ such that $y_i\gQ 0$ for all $i=1,...,m$.
\end{lemma}
	
\noindent Since both primal and dual problems are strictly feasible, the duality gap is zero, that is, strong duality holds %(Theorem \ref{theo:strongDual}) 
and the Karush-Kuhn-Tucker conditions are also sufficient \cite[p.~25]{Alizadeh03}, as Thoerem \ref{theo:optCond} states.

%\begin{theorem}[Strong duality] \label{theo:strongDual}
%	Both (\ref{primal}) and (\ref{dual}) problems have optimal solutions, $x^*$ and $y_1^*,...,y_m^*$,  respectively, and \[\inner{e_1}{x^*} =\sum_{i=1}^m \inner{p_i}{ y_i^*}.\]
%\end{theorem}	
	
\begin{theorem}[Optimality conditions]\label{theo:optCond}
	Let $x^*$ and $y_i^*$, $i=1,...,m$ be any points in $\RR^{n}$. The pair $(x^*, \{y_i^*\}_{i=1,...,m})$ is primal-dual optimal if and only if
	\begin{itemize}
		\item {primal feasibility:}  $x^*\leqQ p_i, \, i=1,...,m$;
		\item{dual feasibility:} $\sum_{i=1}^m y_i^* =e_1$ and $y_i^*\geqQ 0, \, i=1,..., m$;
		\item{complementary slackness:} $\inner{p_i-x^*}{y_i^*}=0,\, i=1,..., m$.
	\end{itemize}
\end{theorem}

\noindent The complementary slackness conditions imply
\begin{itemize}
	\item if $y_i^*\gQ 0$ then $x^*=p_i$ (which can happen for a single $i$, and in that case $y_i^*= e_1$ and $y_j^*=0$ for all $j\neq i$);
	\item if $x^*\gQ p_i$ then $y_i^*=0$;
	\item if $p_i-x^*\in\partial\Q$ and $y_i^*\in\partial\Q$ then 
		\begin{equation}\label{yi} 
		\yb_i^*=\frac{y_{i0}^*}{p_{i0}-x_0^*}\left(\xb^*-\pb_i\right)={y_{i0}^*}\frac{\xb^*-\pb_i}{\norm{\pb_i-\xb^*}}.
		\end{equation}
\end{itemize}

\noindent A consequence of Theorem \ref{theo:optCond} is the following characterization of optimality: 
\begin{theorem}\label{theo:convhull}
	$x^*$ is the optimal solution to problem (\ref{primal}) if and only if $x^*\leqQ p_i$ for all $i=1,...,m$, and 
	\begin{equation}\label{convhull}
		\xb^*\in\conv\left(\{\pb_i:\,  \norm{\pb_i-\xb^*}=p_{i0}-x_0^* \} \right).
	\end{equation}
\end{theorem}

\begin{proof}The theorem follows trivially from the optimality conditions in the case when $x^*=p_k$ for some $p_k\in\Pp$. Consider then, that is not the case.
	
		First, consider $x^*$ is optimal. Let $\I=\{i:\,\norm{\pb_i-\xb^*}=p_{i0}-x_0^*\}$, and $y_i^*$, $i=1,...,m$ be an optimal dual solution. From complementary slackness and dual feasibility we conclude 
		\[0=\sum_{i\in\I} \yb_i^* = \sum_{i\in\I} \frac{y_{i0}^*}{p_{i0}-x_0^*}\left(\xb^*-\pb_i\right)\]
		\noindent since $y_i^*=0$ for $i\not\in\I$. From the previous equation we have
		\[\xb^*= \sum_{i\in\I} \alpha_i\pb_i\quad\text{with}\quad \alpha_i =\frac{{y_{i0}^*}/{p_{i0}-x_0^*}}{\sum_{i\in\I} {y_{i0}^*}/{p_{i0}-x_0^*}}.\]
		\noindent Since $\sum_i y_{i0}^*=1$ there must be at least one $\alpha_i\neq 0$, and $y_i^*\geqQ 0$ implies $\alpha_i\geq 0$ for all $i$. Because $\sum_{i=1}^m  \alpha_i =1$ we conclude that (\ref{convhull}) holds.
		
		\smallskip
	 
	 Conversely, suppose we have $x^*$ that is primal feasible and such that $\xb^*$ satisfies (\ref{convhull}). Then 
	 \[\xb^*= \sum_{i\in \I} \alpha_i\pb_i,\quad\text{with}\quad\alpha_i\geq 0,\,\,i\in\I,\quad \text{and}\quad\sum_{i\in \I} \alpha_i =1.\]
	 
	 \noindent We will now build a dual solution that is feasible and together with $x^*$ satisfy complementary slackness. Consider $y_{i0}$, $i=1,...,m$, such that 
	 \begin{equation}\label{alpi}
	 \alpha_i =\frac{{y_{i0}}/{p_{i0}-x_0^*}}{\sum_j {y_{i0}}/{p_{i0}-x_0^*}},
	 \end{equation}
	 
	 \noindent Equations (\ref{alpi}) together with $\sum_i y_{i0}=1$ give the following linear systems of equations
	 \begin{equation*} \left\{ \begin{array}{l}
	 \alpha_1(z_1+\hdots + z_m) = z_1\\
	 \vdots\\
	 \alpha_m(z_1+\hdots + z_m) = z_m\\
	 \sigma_1 z_1 + \hdots + \sigma_m z_m = 1
	 \end{array}\right.\end{equation*}	 
	 
	 \noindent with $z_i = y_{i0}/(p_{i0}-x_0^*)$ and $\sigma_j=p_{i0}-x_0^*$. The last equation implies
	 \[\sigma_1 \alpha_1(z_1+\hdots + z_m) + \hdots + \sigma_m \alpha_m(z_1+\hdots + z_m)= 1\]
	 \[ (z_1+\hdots + z_m)=\frac1{\sigma_1 \alpha_1 + \hdots + \sigma_m \alpha_m}.\]
	 
	 \noindent Therefore
	 \[z_i = \frac{\alpha_i}{\sigma_1 \alpha_1 + \hdots + \sigma_m \alpha_m},\]
	 \noindent that is,
	 \[y_{i0} = \frac{\alpha_i (p_{i0}-x_0^*)}{\sum_j \alpha_j (p_{j0}-x_0^*)}.\] 
	 
	 \noindent We have that $y_{i0}\geq 0$ for $i\in\I$. Setting $y_{i0}=0$ for $i\not\in\I$, and considering
	 \[\yb_i=\frac{y_{i0}}{p_{i0}-x_0^*}\left(\xb-\pb_i\right),\]
	 \noindent for all $i=1,...,m$, we have that $(y_{i0};\, \yb_i)$ is feasible for the dual problem (\ref{dual}). Since $x^*$ is primal feasible and $\norm{\pb_i-\xb^*}=p_{i0}-x_0^*$ for $i\in\I$, we have that, together with $y_i$, complementary slackness is satisfied, thus $x^*$ is optimal for~(\ref{primal}).
\end{proof}

We now present two insightful conclusions from the proof of Theorem \ref{theo:convhull}.

\begin{observation} 
	The optimal solution $x^*$ to (\ref{primal}) is such that 
	\begin{equation}\label{riconv}
	\xb = \ri \conv (\{\pb_i:\, y^*_{i0}> 0\}).
	\end{equation}
	\noindent Note that $y^*_{i0}>0$ cannot be replaced by $y^*_i\in\partial \Q$ (when $\xb^*=\pb_k$ we have $y_k^*=e_1\not\in\partial\Q$).
\end{observation}

\medskip

\noindent Another conclusion from Theorem \ref{theo:convhull}, which will be at the core of our algorithm, is the following corollary. 
	
\begin{corollary} \label{corol:affhull}
Consider $x$, not necessarily primal feasible, that satisfies
	\begin{equation}\label{eq:observ}
	\xb\in\aff(\{\pb_i:\, i\in\I \}),\quad \text{for}\,\, \I =  \{i:\,\norm{\pb_i-\xb}=p_{i0}-x_0\}.
	\end{equation}
	\noindent Let $\alpha_1,...,\alpha_m$ be the coefficients of the affine combination. There exists a dual solution given by
	\[y_i=0,\, \text{for }\, i\not\in\I,\]
	\[y_{i0} = \frac{\alpha_i (p_{i0}-x_0)}{\sum_j \alpha_j (p_{j0}-x_0)}\,\,\text{and}\, \,\,\yb_i=\frac{y_{i0}}{p_{i0}-x_0}\left(\xb-\pb_i\right),\, \text{for }\, i\in\I, \]
	
	\noindent that satisfies the dual constraint $\sum_{i=1}^m y_i = e_1$, and, together with $x$, the complementary slackness conditions. 
\end{corollary}

\noindent Note that, unless $\{\pb_i:\, i\in\I\}$ is affinely independent, the coefficients of the affine combination are not unique and therefore $x$ may correspond to more than one such dual solution. If, additionally, $\xb\in\conv(\{\pb_i:\, i\in\I \})$, then $y_{i0}\geq 0$ for all $i$, and so $x$ corresponds to a dual feasible solution. 
\bigskip 

Finally, we can use the result from Theorem \ref{theo:convhull} to easily calculate algebraically the solution to $\InfQP$ when $\Pp$ only has two points.
\begin{theorem}\label{theo:2pts}
	The solution to $Inf_\Q(\{p_1, p_2\})$ is given by
	\begin{equation}\label{eq:2pts}
	x_0^* = \min\left(p_{10}, p_{20}, \frac{p_{10}+ p_{20} -\norm{\pb_1-\pb_2}}{2}\right),\,
	\xb^* = \frac{(p_{10}-x_{0}^*)\pb_{2} +  (p_{20}-x_{0}^*)\pb_{1} }{(p_{10}-x_{0}^*) + (p_{20}-x_{0}^*)}.
	\end{equation}
\end{theorem}

\subsection{The definition of \emph{support set} and \emph{S-pair}}

We now define \emph{support set} for the $Inf_\Q(\Pp)$ problem the same way a basis is defined for an LP-type problem \cite{Dyer04}.

\begin{definition}[{Support set}]\label{def:suppset}
	A subset $\Ss$ is called a \emph{support set} if no proper subset $\Ss'$ of $\Ss$ is such that $Inf_\Q(\Ss')=Inf_\Q(\Ss)$. A subset $\Ss\subseteq \Pp$ is said to be an \emph{optimal support set} if 	$\Ss$ is a support set and $Inf_\Q(\Ss)=Inf_\Q(\Pp)$.
\end{definition}

\noindent Note that an optimal support set may not be unique.
\begin{definition}[{Dual feasible S-pair}]\label{def:dualSpair}
	Let $\Ss\subseteq\Pp$ and $x$ a vector. We say $(\Ss, x)$ is a \emph{dual feasible S-pair} if
	\begin{enumerate}[(a)]
		\item the points of $\Ss$ lie on the boundary of $x+\Q$, that is, 
		\[\norm{\pb_i-\xb} = p_{i0}-x_0,\quad \forall\,p_i\in\Ss;\]
		\item $\xb\in\ri\conv(\Sb)$;
		\item  $\Sb$ is affinely independent.
	\end{enumerate}
\end{definition}

Comparing with LP, \emph{support set} is analogous to the definition of \emph{basis}, and \emph{dual feasible S-pair} to the definition of \emph{dual basic feasible solution}. 

Using Theorem \ref{theo:convhull}, it is now possible to prove the following equivalence.

\begin{theorem}\label{theo:suppset1}
	$(\Ss, x)$ is a dual feasible S-pair iff $\Ss$ is a support set and $x$ is the solution to $\InfQS$. 
\end{theorem}

\smallskip

As a consequence of the previous theorem, a support set has at least $1$ point and at most $n$ points. This is, in fact, a well known result: the minimum enclosing ball of balls in $\RR^{n-1}$ is defined by at most $n$ balls. Another consequence, is the fact that if $(\Ss, x)$ is a dual feasible S-pair and $x$ is primal feasible, then $x$ solves $\InfQP$.

\medskip   

Finally, Theorem \ref{theo:suppset4} is a direct consequence of the definition of support set.
\begin{theorem}\label{theo:suppset4}
	Let $\Ss$ be a support set. If a point $p^*$ is infeasible to $Inf_\Q(\Ss)$, then $p^*$ belongs to an optimal support set for problem $Inf_\Q(\Ss\cup\{p^*\})$.
\end{theorem}

\section{Algorithm description}\label{sec:algorithm} We now outline the basic framework of the algorithm:
\smallskip

\begin{description}
	\item[Initialization:] Suppose a dual feasible S-pair $(\Ss^0, x^0)$, $\Ss^0\subseteq\Pp$, is given.
	\item[Loop:] For $j=0,1,2,...$, do:
		\begin{enumerate}[(a)]
			\item If $x^j$ is primal feasible, stop - $x^j$ is the optimal solution of (\ref{primal}).
			\item Else, get a $p^*\in\Pp$ corresponding to a chosen violated constraint.
			\item Obtain a new dual feasible S-pair $(\hat{\Ss}\cup\{p^*\}, \hat{x})$, for $\hat\Ss\subseteq\Ss^j$ and $\hat x_0<x_0^{j}$; 
			
			Set $(\Ss^{j+1}, x^{j+1})\leftarrow (\hat\Ss\cup\{p^*\}, \hat x)$.
		\end{enumerate}
\end{description}

\noindent Note that, as an initial solution, one can simply pick any point $p\in\Pp$, and consider the dual feasible S-pair $(\{p\}, p)$.

\smallskip

The core of the algorithm is step (c) which will consist on a sequence of \emph{curve searches} that keep dual feasibility. Assume, for now, that $\Sb^j\cup\{\pb^*\}$ is affinely independent. Since, at each iteration $j$, we start with a dual feasible S-pair $(\Ss^j, x^j)$, the set of constraints associated with $\Ss^j$ are active at $x^j$. As we shall see, the set of points $x$ that correspond to dual feasible solutions and that keep those constraints active define a curve (see section \ref{subsec:curve}). A \emph{curve search} basically consists on the following problem: starting at $x^j$, we ``move'' on the curve in the direction of decrease of $x_0$ until either dual feasibility is lost or $p^*$ becomes feasible, whichever happens first. As we will see, we will be able to calculate exactly the points where dual feasibility is lost and where $p^*$ becomes feasible, so the \emph{curve search} will be \emph{exact}. This then boils down to calculating what we will define as the \emph{partial step} and the \emph{full step}:

\smallskip
\noindent \emph{\textbf{Partial step}}: the maximum step on the curve without violating dual feasibility. The point corresponding to that step, which we will denote by $x^{partial}$, is the first point on the curve where one of the dual variables, which vary non-linearly as we move on the curve, becomes zero.

\smallskip
	
\noindent\emph{\textbf{Full step}}: the minimum step on the curve such that the constraint corresponding to $p^*$ is feasible, that is, when it becomes active. When it exists, we denote that point by $x^{full}$.	

{\medskip}

\noindent If the full step happens before the partial step, then $x^{full}$ corresponds to a dual feasible variable and so it is the optimal solution to the subproblem $\InfQSjp$. In this case, a major iteration is complete, and we go back to Step (a) after setting $(\Ss^{j+1}, x^{j+1})\leftarrow(\Ss^j\cup\{p^*\}, x^{full})$. Otherwise, if the partial step happens first, then the point (say $p_k\in\Ss^j$) corresponding to the dual variable that was about to become infeasible is dropped from $\Ss^j$. A new curve search starting at $x^{partial}$ is then performed. Eventually, after at most $|\Ss^j|$ curve searches where the value of $x_0$ either decreases or stays the same, the optimal solution of $\InfQ(\Ss'\cup\{p^*\})$, for some $\hat\Ss\subseteq \Ss^j$, with a strictly smaller objective function is found.

\smallskip

A curve search is only performed whenever $\Sb^j\cup\{\pb^*\}$ affinely independent. When that is not the case, a point is immediately dropped from $\Ss^j$ before a curve search is done (this step can actually be seen as a partial step with zero length). Section \ref{subsec:curve} defines the curve and section \ref{subsec:search} explains the details of the curve search.

\subsection{The curve}\label{subsec:curve}

Consider iteration $j$. Consider
\begin{itemize}
	\item $(\Ss^j, x^j)$, a dual feasible $\Ss$-pair for $\InfQSj$, with $\Ss^j\subseteq \Pp$;
	\item $p^*\in\Pp$ corresponding to an infeasible constraint at $x^j$;
	\item  $\Sb^j := \{\pb:\, p\in\Ss^j\}$ and $s = |\Ss^j|$.
\end{itemize}

\smallskip

The algorithm restricts the search for the next iterate to the set of primal solutions $x$ where dual feasibility and complementary slackness are maintained, that is
\begin{align}
\norm{\pbij-\xb}= \pijo -x_0, &\quad \forall\, \pij\in\Ss^j,\label{eq:curveCones}\\
\xb\in\convSjp.&\label{eq:curveConv}
\end{align}
\noindent while it decreases the objective function value $x_0$. 

In this section we shall see that in general the set of points that satisfy (\ref{eq:curveCones}) and
\begin{equation}\label{eq:curveAff}
\xb\in\affSjp
\end{equation}
\noindent constitute a manifold of dimension $1$, that is, a curve.

\smallskip
By squaring equations (\ref{eq:curveCones}) it is easy to prove the following Lemma:
\begin{lemma}\label{lem:curve}
	If $|\Ss^j|>1$, define
	\begin{align*}
	&M=\left[\begin{array}{cccc} \pb_{j_2}-\pb_{j_1}& \pb_{j_3}-\pb_{j_1}&...& \pb_{j_s}-\pb_{j_1}\end{array}\right],\\
	&c = \left(\begin{array}{c}p_{{j_2}0}-p_{{j_1}0}\\\vdots\\ p_{{j_s}0}-p_{{j_1}0}\end{array}
	\right),   
	\,\text{and}\,\,
	b = \frac12\left(\begin{array}{c}\norm {\pb_{j_2}}^2 - p_{{j_2}0}^2 -\norm{\pb_{j_1}}^2 + p_{{j_1}0}^2\\\vdots\\ \norm {\pb_{j_s}}^2 - p_{{j_s}0}^2 -\norm {\pb_{j_1}}^2 + p_{{j_1}0}^2\end{array}	\right).
	\end{align*}
	\noindent Otherwise, when $|\Ss^j|=1$, $M=[\;]$, $b=[\;]$, and $c=[\; ]$.
	
	Conditions (\ref{eq:curveCones}) are equivalent to the following conditions
	\begin{subequations}\label{eq:curve}
		\begin{align}
		&M^T \xb = b +x_0c,\label{eq:curve1}\\
		&\norm{\pb_{j_1}-\xb}^2 = \left(p_{{j_1}0}-x_0\right)^2,\label{eq:curve2}\\
		&x_0 \leq \min_{p_{j_i}\in\Ss^j}\{p_{{j_i}0}\}.\label{eq:curve3}
		\end{align}	
	\end{subequations}
\end{lemma}

%\begin{proof}We square equations (\ref{eq:curveCones}) obtaining the equivalent conditions
%	\[	\norm{\pb_{j_i}-\xb}^2 - \left(p_{{j_i}0}-x_0\right)^2=0 \quad\text{and}\quad p_{j_i}-x_0\geq 0,\quad p_{j_i}\in\Ss^j,\]
%	\noindent which are equivalent to the system of equations and one inequality
%	\begin{equation}
%	\norm{\pb_{j_i}-\xb}^2 - \left(p_{{j_i}0}-x_0\right)^2 = \norm{\pb_{j_1}-\xb}^2 - \left(p_{{j_1}0}-x_0\right)^2,\quad p_{j_i}\in\Ss^j\setminus\{p_{j_1}\},\label{aux1}
%	\end{equation}
%	\[\norm{\pb_{j_1}-\xb }^2 = (p_{{j_1}0}-x_0 )^2\]
%	\[x_0 \leq \min_{p_{j_i}\in\Ss^j}\{p_{{j_i}0}\}.\]
%	After several simplifications (\ref{aux1}) is equivalent to 
%	\[-(p_{j_i0}-p_{j_10})x_0 + (p_{j_i}-p_{j_1})^T\xb =  \frac12\left(\norm {\pb_{j_i}}^2 - p_{{j_i}0}^2 -\norm {\pb_{j_1}}^2 + p_{{j_1}0}^2\right), \quad p_{j_i}\in\Ss^j,\]
%	\noindent that is, $-x_0c + M^T\xb = b$.
%	\end{proof}

\medskip

For the general case of $|\Ss^j|>1$, let us now define the following matrix and vectors
\begin{align*}
&M^+=\,(M^TM)^{-1}M^T,\\
&u\,=\,(M^TM)^{-1}(b-M^T\pb_{j_1}),\\
&v\,\,=\,(M^TM)^{-1}c, \\
&w\,=\,-M^+(\pb^*-\pb_{j_1}),\\
&z\,\,\,=\, (I-MM^+)(\pb^*-\pb_{j_1}).% = Mw + (\pb^*-\pb_{j_1}) .
\end{align*}
\noindent Matrix $M^+$ is the Moore-Penrose inverse, or pseudo-inverse, of $M$. Since $M$ is full column rank, $M^+$ is a left-inverse ($M^+M=I$). When $|\Ss^j|=1$, simply consider 
\[u=v=w=[\;]\quad\text{and}\quad z=\pb^*-\pb_{j_1}\]

\begin{theorem}\label{theo:paramCurve}
	Conditions (\ref{eq:curveCones}) and (\ref{eq:curveAff}) define points $(x_0;\xb)$ such that
	\begin{equation}\label{eq:curvelem1}
	\xb= M(u+ x_0v)  +\alpha^* z+ \pb_{j_1},
	\end{equation}
	
	\noindent for $x_0\leq \min_{p_{j_i}\in\Ss^j}\{p_{{j_i}0}\}$, $\alpha^*\in \RR$, and $x_0$ and $\alpha^*$ such that
	\begin{equation}\label{eq:curvelem2}
	(\alpha^*)^2\norm{z}^2 + \norm{M(u+ x_0v)}^2 - (p_{10}-x_0)^2 =0.
	\end{equation}
	
\end{theorem}

\begin{proof} From (\ref{eq:curveAff}) we know that there exist $\alpha_1,\dots,\alpha_s,\alpha^*$ such that
	\[\xb = \sum_{i=1}^s \alpha_i\pb_{j_i} + \alpha^*\pb^*\quad\text{and}\quad\sum_{i=1}^s\alpha_i +\alpha^* =1,\]
	\noindent that is,
	\begin{equation}\label{eq:aux1}
	\xb = M\alpha_{2:s}	+ \alpha^*(\pb^*-\pb_{j_1}) + \pb_{j_1}
	\end{equation}
	\noindent with $\alpha_{2:s}$ the vector with $\alpha_2,\dots,\alpha_s$.
	Substituting in (\ref{eq:curve1}), we have
	\[M^TM\alpha_{2:s} + \alpha^*M^T(\pb^*-\pb_{j_1}) + M^T\pb_{j_1} = b+x_0c,\]
	\noindent which yields
	\begin{equation}\label{eq:aux2}
	\alpha_{2:s} =	u + x_0v + \alpha^*w,
	\end{equation}
	\noindent with $u$, $v$, and $w$ as defined above. Combining (\ref{eq:aux1}) and (\ref{eq:aux2}) we obtain (\ref{eq:curvelem1}). Finally, by plugging (\ref{eq:curvelem1}) in  (\ref{eq:curve2}) we obtain (\ref{eq:curvelem2}).
\end{proof}

\subsubsubsection{The curve: the affinely independent case}

\noindent When $\Sb^j\cup\{\pb^*\}$ is affinely independent then we always have $z\neq 0$, thus the quadratic equation (\ref{eq:curvelem2}) can be solved for $\alpha^*$, the coefficient of the affine combination associated with $p^*$, obtaining
\begin{equation}\label{eq:alphastarPre}
\alpha^*(x_0)=\pm\frac1{\norm{z}}\sqrt{(p_{10}-x_0)^2 -  \norm{M(u+ x_0v)}^2}.
\end{equation}
We then get the points that satisfy (\ref{eq:curveCones},\ref{eq:curveAff}) as a function of a single variable $x_0\in[-\infty, \min_{p_{j_i}\in\Ss^j}\{p_{{j_i}0}\}]$, thus defining a curve $\Gamma$. This is shown by Corollary \ref{corol:paramCurve}.

\begin{corollary}\label{corol:paramCurve}
	Consider $\Sb^j\cup\{\pb^*\}$ affinely independent. Then, conditions (\ref{eq:curveCones},\ref{eq:curveAff}) define a curve parameterized by 
	\[x_0\in \left[-\infty, \min_{p_{j_i}\in\Ss^j}\{p_{{j_i}0}\}\right],\] 
	\noindent such that 
	\begin{equation*}\label{eq:curvetheo4}
		\xb	=\Gamma (x_0) := M(u+ x_0v)  \pm\frac z{\norm{z}}\sqrt{(p_{10}-x_0)^2 -  \norm{M(u+ x_0v)}^2} + \pb_{j_1}.
	\end{equation*}
\end{corollary}

\noindent Geometrically, the curve is either a hyperbola (Figure~\ref{fig:curveA}a) or a degenerate hyperbola, see Figure~\ref{fig:curveA}b). The curve is also symmetric with respect to a reflection through the hyperplane $\{(x_0;\xb):\, \xb\in\affSjp,\,x_0\in\RR\}$ and intersects the boundary of $\conv(\Sb^j\cup\{\pb^*\})$ at two points (one of them in $\conv(\Sb^j)$).

\begin{center}\begin{figure}[h] \centering
		\subfigure[Case when $|\Ss^j|=2$.]{\includegraphics[height=3cm]{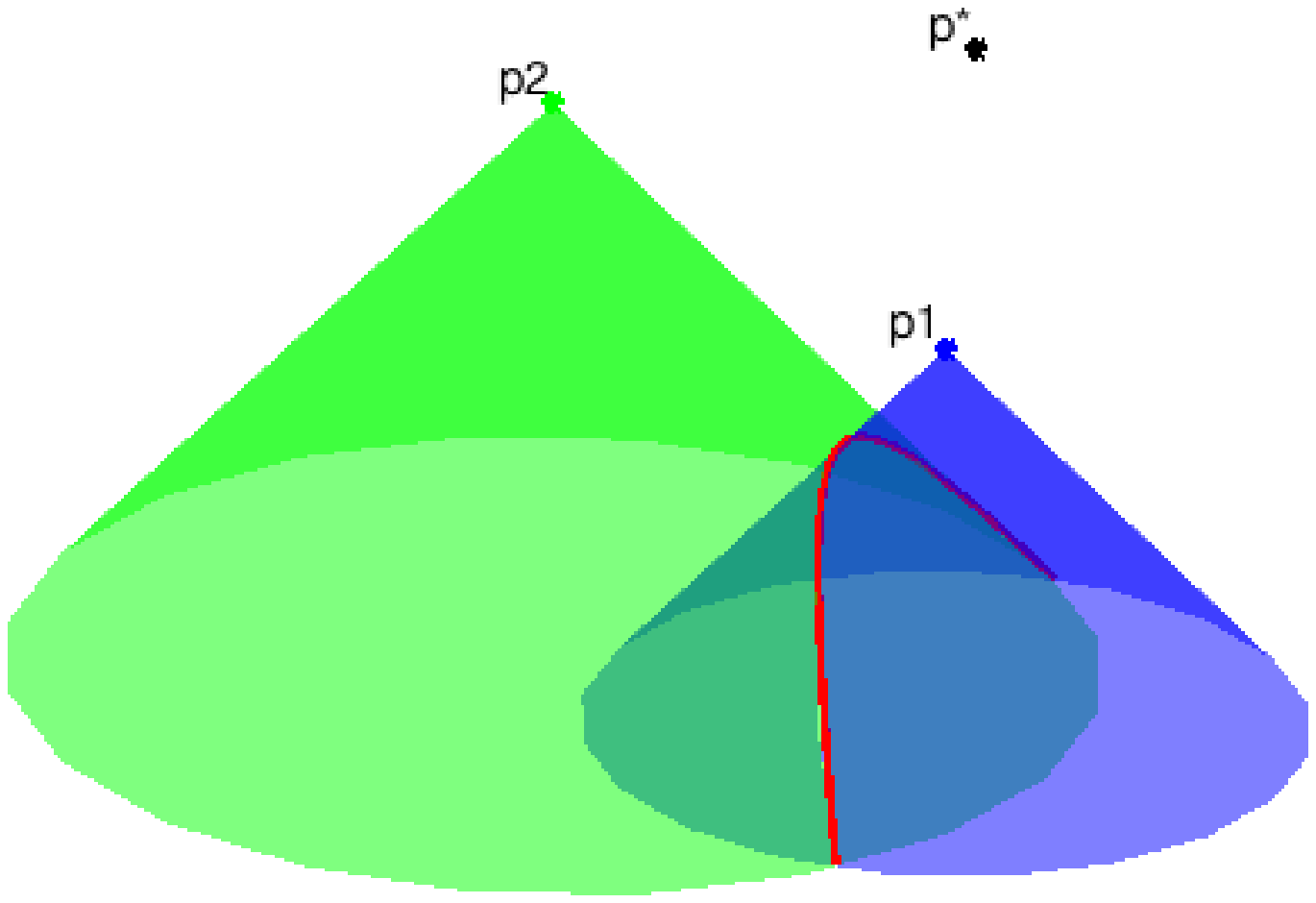}}
		\hspace{2cm}
		\subfigure[Case when $|\Ss^j|=1$.]{\includegraphics[height=3cm]{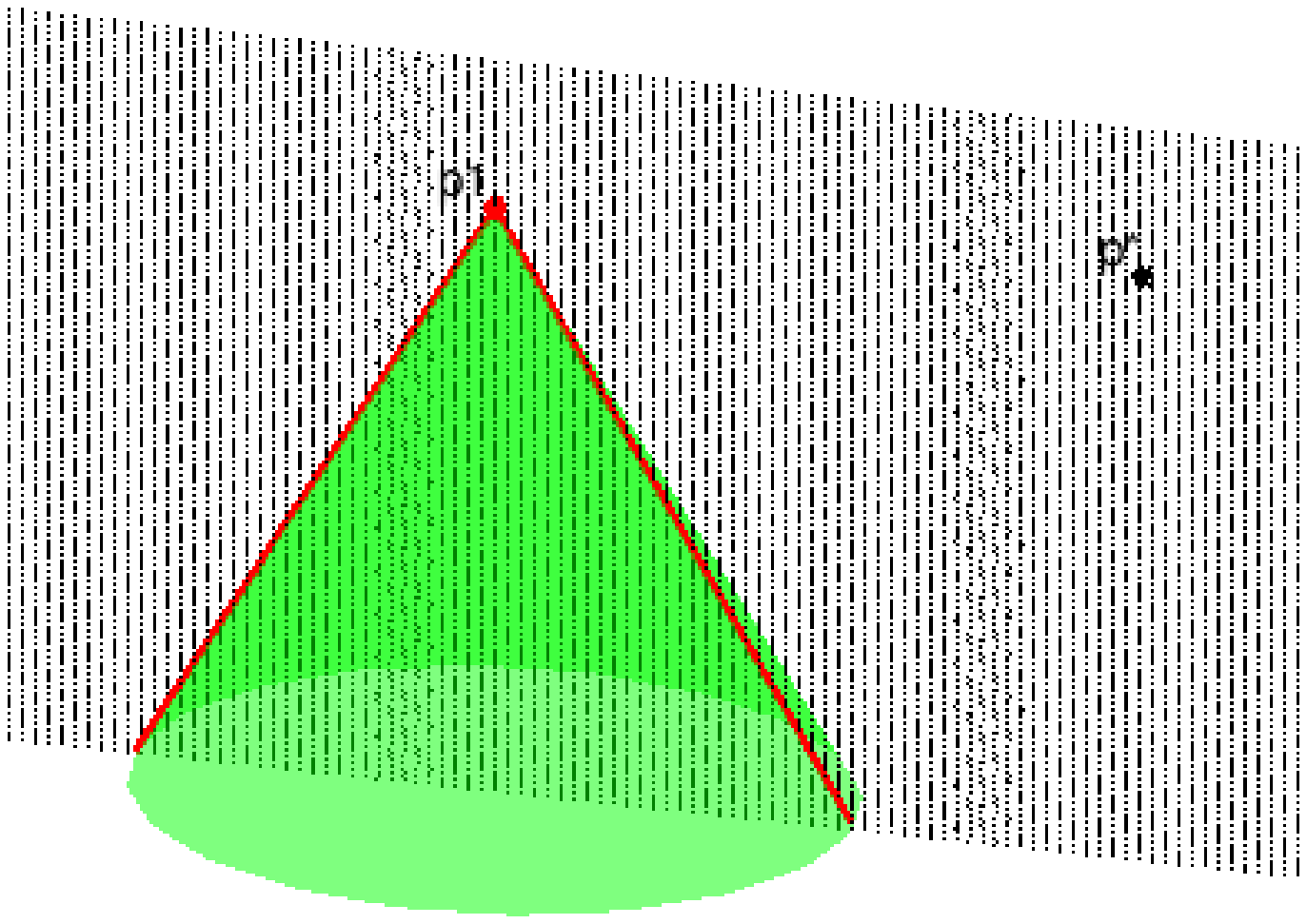}}
		\vspace{-0.3cm}
		\caption{Illustration of the curve in $\RR^3$.}\label{fig:curveA}	
	\end{figure}\end{center}

When we are interested in the part of the curve that corresponds to a dual feasible solution, we need all coefficients of the affine combination to be non-negative. Therefore the portion of the curve we are interested on corresponds to $\alpha^*(x_0)\geq 0$, that is
\[\xb=\Gamma^+ (x_0) := M(u+ x_0v)  +\frac z{\norm{z}}\sqrt{(p_{10}-x_0)^2 -  \norm{M(u+ x_0v)}^2} + \pb_{j_1}.\]

\medskip

Corollary \ref{corol:paramAlphas} is a direct consequence of Theorem \ref{theo:paramCurve}. 

\begin{corollary}\label{corol:paramAlphas}
	Consider $\Sb^j\cup\{\pb^*\}$ affinely independent. The points $x$ that satisfy (\ref{eq:curveCones},\ref{eq:curveAff}) and for which $\alpha^*(x_0)\geq 0$ can be written as an affine combination
	\begin{equation}\label{eq:curveAlphas}
	\xb = \sum_{i=1}^s \alpha_i(x_0)\pb_{j_i} + \alpha^{*+}(x_0)\pb^*,
	\end{equation}
	\noindent for $x_0\in [-\infty, \min_{p_{j_i}\in\Ss^j}\{p_{{j_i}0}\}]$, such that
	\begin{subequations}\label{eq:alphasx0}
		\begin{align}
		&\alpha^{*+}(x_0)= \frac1{\norm{z}}\sqrt{(p_{10}-x_0)^2 -  \norm{M(u+ x_0v)}^2},\label{eq:alphasx0star}\\
		&\alpha_1(x_0)=1 - {1}_{s-1}^T(u+x_0v) -\alpha^{*+}(x_0)\left({1}_{s-1}^Tw+1\right),\label{eq:alphasx01}\\
		%&\alpha_1(x_0)=1 - {1}_{s-1}^T(u+x_0v) -\frac 1{\norm{z}}\sqrt{(p_{10}-x_0)^2 -  \norm{M(u+ x_0v)}^2}\left({1}_{s-1}^Tw+1\right)\\
		&\alpha_{2:s}(x_0)=u + x_0v +\alpha^{*+}(x_0)\frac w{\norm{z}},\label{eq:alphasx02s}
		%&\alpha_{2:s}(x_0)=u + x_0v +\sqrt{(p_{10}-x_0)^2 -  \norm{M(u+ x_0v)}^2}\frac w{\norm{z}}
		\end{align}
	\end{subequations}
	\noindent with $1_{s-1}$ a $(s-1)-$vector with entries all $1$, and $\alpha_{2:s}$ the vector with $\alpha_2,\dots,\alpha_s$.
	%	\begin{enumerate}[(a)]
	%		
	%		\item $\alpha_i$, $i=1,...,s$ are functions of $\alpha^*\geq 0$:
	%		\begin{eqnarray}\label{eq:alphasaSt}
	%		\alpha_1(\alpha^*)&=&1 - \alpha^*- {1}_{s-1}^T\left(u+x_0(\alpha^*)v-\alpha^*w\right) \\
	%		\alpha_{2:s}(\alpha^*)&=&u + x_0(\alpha^*)v +\alpha^*w
	%		\end{eqnarray}
	%		\noindent with $x_0(\alpha^*)$ as defined in part (a) of Theorem \ref{theo:paramCurve};
	
	%\end{enumerate}
\end{corollary}

\noindent Using the formulas from Corollary \ref{corol:affhull}, it is now possible to write explicitly the dual variables as functions of $x_0$ corresponding to each point $(x_0;\Gamma^+(x_0))$ on the curve.

%\begin{align*}
%	&\yijo(x_0) =\frac{\alpha_i(x_0) (\pijo-x_0) }{\sum_{k=1}^s \alpha_k(x_0)(\pjko-x_0) + \alpha^*(x_0)(p^*_0-x_0)}, \;
%	\ybij(x_0) = -\frac{\yijo(x_0)}{\pijo-x_0}\left(\pbij-\Gamma^+(x_0)\right), \quad \pij\in\Ss^j,\\	
%	&y_{0}^*(x_0) =\frac{\alpha_i(x_0)(p^*_0-x_0)}{\sum_{k=1}^s \alpha_k(x_0)(\pjko-x_0) + \alpha^*(x_0)(p^*_0-x_0)}    , \;
%	\yb^*(x_0) = -\frac{y_{0}^*(x_0)}{p^*_0-x_0}\left(\pb^*-\Gamma^+(x_0)\right),\\
%	& y_k=0,\quad p_k\not\in\Sjp
%\end{align*}

\subsubsubsection{Case when $\Sb^j\cup\{\pb^*\}$ is affinely dependent}

\noindent When $\Sb^j\cup\{\pb^*\}$ is affinely dependent, the linear system (\ref{eq:curve1}) together with condition (\ref{eq:curveAff}) define an affine space of dimension $1$, that is, a line, which intersects the manifold defined by (\ref{eq:curve2}-\ref{eq:curve3}), on a single point, $x^j$. That is because $x^j$ satisfies both (\ref{eq:curveCones}) and $\xb^j\in\convSj=\convSjp$ if and only if $x^j$ is the solution to $\InfQSj$, which is unique. 

\begin{theorem}\label{theo:isNotAcurve}
	If $\Sbjpb$ is affinely dependent then conditions (\ref{eq:curveCones},\ref{eq:curveAff}) have as solution the single point $x^j$ that corresponds to a dual feasible solution.
\end{theorem}

\noindent When $\Sb^j\cup\{\pb^*\}$ is affinely dependent, Theorem~\ref{theo:paramCurve} still holds with  $z=0$. In particular (\ref{eq:aux1}) also still holds. But in this case $\alpha_i$, $i=1,...,s$, and $\alpha^*$ cannot be written as functions of $x_0$ as in Corollary \ref{corol:paramAlphas}, but they can be written instead in terms of $\alpha^*$. Since, in this case, $x_0(\alpha^*)=x_0^j$ for any value of $\alpha^*$, from (\ref{eq:aux2}) we have
\begin{subequations}\label{eq:alphasAD}
	\begin{align}
	&\alpha_{2:s}(\alpha^*) = u+x_0^jv+\alpha^* w. \label{eq:alphasAD2s}\\	
	&\alpha_1(\alpha^*) = 1-\alpha^* - 1^T(u+x_0^jv+\alpha^* w), \label{eq:alphaAD1}
	\end{align}
\end{subequations}

\subsection{The curve search}\label{subsec:search}

At the beginning of a curve search step, let $\Ss^j$ be the support set and $x^j$ the current iterate. If we are starting a major iteration (we come from step (b)) then $(\Ss^j,x^j)$ is a dual feasible S-pair, otherwise, if we come from a previous curve search, it is not. Regardless of the case, at the beginning of a curve search we always have the following
\begin{itemize} 
	\item $\Sb^j\cup\{\pb^*\}$ is affinely independent;
	\item $\xb\in\convSjp$.
\end{itemize}
\noindent This fact will be proved in section \ref{subsec:correctness}.

\medskip

The curve search consists on, starting at $x^j$, moving on the curve defined by $\Gamma^+(x_0)$ in the direction of decrease of $x_0$, until either dual feasibility is lost or the violated constraint becomes feasible. Using the conclusions from the previous section, that corresponds to solving the following problem
%\begin{equation*} \begin{array}{lll}
%\min & x_0&\\
%\st&\norm{\pb^*-\xb} \geq p_0^*-x_0&\\
%& 	\norm{\pbij-\xb}= \pijo -x_0 &\quad \forall\, \pij\in\Ss^j,\\
%&	\xb\in\convSjp.&
%\end{array}\end{equation*}
\begin{equation*} \begin{array}{ll}
\min & x_0\\
\st&\norm{\pb^*-\xb} \geq p_0^*-x_0,\\
& 	\alpha_i(x_0)\geq 0, \quad i=1,...,s, \\
&	\alpha^*(x_0)\geq 0,\\
& x_0\leq \min_{p_{j_i}\in\Ss^j}\{p_{{j_i}0}\}.
\end{array}\end{equation*}

\noindent Note that since $\xb^j\in\convSjp$ we know that $\alpha_i(x_0^j)\geq 0$, for all $i=1,...,s$, and $\alpha^*(x_0^j)\geq 0$. 

Therefore, the solution to the problem above, is always smaller or equal to $x_0^j$ which always satisfies the last inequality.

\bigskip

Unless stated otherwise, for the remainder of this section, we shall consider $|\Ss^j|>1$ and $\Sb^j\cup\{\pb^*\}$ affinely independent.

\subsubsubsection{The partial step}

\noindent The partial step consists on finding the point corresponding to the maximum step on the curve that can be made without losing dual feasibility. When $\Sbjpb$ is affinely independent, Corollary \ref{corol:paramAlphas} describes the dual variables as we move on the curve. Therefore, the point where dual feasibility is lost corresponds to the value of $x_0$ that solves
\begin{equation} \begin{array}{rl}\label{partialStep}
\min & x_0\\
\st & \alpha_i(x_0)\geq 0,\quad i=1,...,s,\\
&\alpha^*(x_0)\geq 0,\\
&x_0\leq x_0^j.
\end{array}\end{equation}	

\noindent Note that, if we consider the definitions of $\alpha_i$, $i=1,...,m$ given by (\ref{eq:alphasx0}), the constraint $\alpha^*(x_0)\geq 0$ is already implied. The solution to (\ref{partialStep}) is then found by simply solving the $s$ equations (\ref{eq:alphasx01}-\ref{eq:alphasx02s}) subject to the inequality $x_0\leq x_0^j$ and picking the largest of the solutions. These are radical equations that can be easily solved by isolating the square root term in one side, squaring both sides, solving the squared equation, and finally discarding any extraneous solutions. Note that, for some $i$ the corresponding equation may not have a solution, since the curve may not intersect the supporting hyperplane of the corresponding facet of $\convSjp$. 

The partial step procedure is summarized in Algorithm \ref{alg:partial}. Note that this procedure will only be done when $|\Ss^j|>1$, as we shall see ahead.

\begin{algorithm}
	\caption{Partial step procedure}\label{alg:partial}
	\begin{algorithmic}[1]
		\Require Matrix $M$, vectors $u,v,w,z$, point $x^j$ and set $\Ss^j$.
		\Ensure $x_0^{partial}$ and index $k$.		
		\smallskip
		\For {$i=1,...,s$}
		\State Define $\alpha_i(x_0)$ as in (\ref{eq:alphasx0});
		\State $\hat\A_i\gets$ solutions to $\alpha_i(x_0)=0$;
		\State $\A_i \gets \left\{\delta\in\hat\A_i:\, \delta\,\text{is real}\,\wedge\, \delta\leq x_0^j\right\}$;
		\EndFor		
		\State $\displaystyle x_0^{partial}\gets \max \left\{\cup_{i=1,...,s}\,\A_i\right\}$;
		\State $k\gets i \;\,\st\, x_0^{partial}\in\A_i$.
	\end{algorithmic}
\end{algorithm}

\subsubsubsection{The full step}

\noindent The full step procedure finds the point corresponding to the minimum step on the curve for which $p^*$ is feasible, that is, the first point on the curve where the primal constraint corresponding to $p^*$ is feasible is active:
\begin{equation}\label{eq:pfeas1}
\norm{\pb-\xb} = p_0^*-x_0.
\end{equation}
\noindent The system of equations (\ref{eq:curve1},\ref{eq:curve2},\ref{eq:pfeas1}) is equivalent to the system of equations (\ref{eq:curve1},\ref{eq:curve2},\ref{eq:pfeas2}), with (\ref{eq:pfeas2}) being
\begin{equation}\label{eq:pfeas2}
(\pb^*-\pb_{j_1})^T \xb = b^* +x_0c^*, 
\end{equation}
\noindent such that 
\[c^* = p^*_0-p_{{j_1}0}\quad \text{and}\quad b^* = \frac12\left(\norm {\pb^*}^2 - (p^*_{0})^2 -\norm{\pb_{j_1}}^2 + p_{{j_1}0}^2\right),\]
\noindent and for $x_0$ satisfying (\ref{eq:curve3}) and $x_0\leq p^*_0$.
\noindent Therefore, the full step consists on solving
\begin{equation} \begin{array}{rl}\label{fullStep}
\max & x_0\\
\st &\xb = \Gamma^+(x_0)\\
&(\pb^*-\pb_{j_1})^T \xb = b^* +x_0c^*\\
&x_0\leq {p^*_0}\\
&x_0\leq x_0^j.
\end{array}\end{equation}

\noindent Note that the solution to (\ref{fullStep}) to may not exist. It is easy to see that equations (\ref{eq:curveCones},\ref{eq:curveAff}) together with (\ref{eq:pfeas1}) seek the position of the vertex of a second-order cone that has points $\Sb^j\cup\{p^*\}$ on its boundary, which is not guaranteed to exist. Additionally, such point may not be unique.

To solve (\ref{fullStep}), we could plug $\xb=\Gamma^+(x_0)$ in (\ref{eq:pfeas1}), and solve for $x_0$. However we realized this approach would result in rather intricate calculations involving square roots. So, instead, we go back to the results of Lemma \ref{lem:curve} and Theorem \ref{theo:paramCurve} and proceed as we explain next. We start by plugging (\ref{eq:curvelem1}) in (\ref{eq:pfeas2}), and solve for $\alpha^*$. We then obtain $\alpha^*$ in terms of $x_0$:
\begin{equation}\label{eq:fsalphaStar}
\alpha^*(x_0) = \frac1{(\pb^*-\pb_{j_1})^Tz} (b^* +x_0c^* - (\pb^*-\pb_{j_1})^T (M(u+ x_0v) + \pb_{j_1})),
\end{equation}
\noindent Now, we could solve (\ref{eq:curvelem2}) for $x_0$. However, the following is easier: we plug (\ref{eq:fsalphaStar}) back in (\ref{eq:curvelem1}), which now allows $\xb$ to be written solely as a linear function of $x_0$ simply as
\[\xb(x_0) = q + x_0r+\pb_{j_1},\] 
\noindent with
\[q=Mu+\frac{b^*-(\pb^*-\pb_{j_1})^T(Mu+\pb_{j_1})}{(\pb^*-\pb_{j_1})^Tz}\myand r=Mv+\frac{c^*-(\pb^*-\pb_{j_1})^T(Mv)}{(\pb^*-\pb_{j_1})^Tz},\]

\noindent and we solve (\ref{eq:curve2}). This results in a much simpler quadratic equation on $x_0$: 
\begin{equation}\label{eq:fseq}
\norm{q+x_0r}^2 = (p_{j_10}-x_0)^2.
\end{equation}

\noindent Note that, in this process, we never imposed that $\alpha^*(x_0)\geq 0$. If the solution(s) to (\ref{eq:fseq}) are real, we only keep the one(s) that satisfy both
\begin{equation*}
\alpha^*(x_0)\geq 0\quad\text{and}\quad x_0\leq \min\left\{x_0^j, p_0^*\right\},
\end{equation*}
\noindent for $\alpha^*(x_0)$ given by (\ref{eq:fsalphaStar}). If there are more than one such solution, we pick the one with the maximum value of $x_0$ (the one that ``occurs first'' as we move the curve), which will become $x_0^{full}$, the solution of (\ref{fullStep}). If we are left with no solutions or the solutions were not real, then, as explained previously, that means that such point does not exist. In this case, we shall consider $x_0^{full}=-\infty$. The details of this procedure are summarized in Algorithm \ref{alg:full}.

\begin{algorithm}
	\caption{Full step procedure}\label{alg:full}
	\begin{algorithmic}[1]
		\Require Matrix $M$, vectors $u,v,w,z$, points $x^j, p^*$, and set $\Ss^j$
		\Ensure $x_0^{full}$.		
		\smallskip	
		\State Define $c^*,b^*, q$ and $r$;
		\State Get $\hat \A^*$ the set of solutions to $\norm{q+x_0r}^2 = (p_{j_10}-x_0)^2$.
		\State $\A^*\gets \{\delta\in\hat\A^*:\, \alpha^*(\delta)\geq 0\,\wedge\, \delta\leq p_0^*\,\wedge\,  \delta\leq x_0\}$, for $\alpha^*$ as in (\ref{eq:fsalphaStar}).
		\If {$\A^* = \emptyset$} 
		\State $x_0^{full}\gets-\infty$ 
		\Else 
		\State $\;x_0^{full}\gets \max \{\A^*\}$\EndIf	
	\end{algorithmic}
\end{algorithm}

\subsubsubsection{Taking a step}

\noindent If $x_0^{partial}< x_0^{full}$, then the $p^*$ becomes feasible while all dual variables associated with $\Ss^j$ are feasible too. That implies that $x^{full}$ is the solution to $\InfQSjp$. The algorithm now starts a new major iteration with a dual feasible S-pair $(\Sj\cup\{p^*\}, x^{full})$.

On the other hand, if $x_0^{partial}\leq x_0^{full}$, then a dual variable became $0$ before $p^*$ was primal feasible. Geometrically, that means the curve hits a facet of $\convSjp$ before hitting the translated cone $\pb^*-\Q$. The point $p_k$ that results from applying (\ref{eq:rs2}) is an opposite point to that hit facet. When this happens, $p_k$ is dropped from $\Ss^j$: $\Ss^j = \Ss^j\setminus\{p_k\}$, and the algorithm then performs a new curve search with the new $\Ss^j$ starting at $x^{partial}$. It may happen that there may be more than one dual variable that became $0$ simultaneously, that is, at line 6 of Algorithm \ref{alg:partial}, $x_0^{partial}$ belongs to more than one $\A_i$. When that happens, after dropping one of such points at the first curve search, the next curve search will consist on a partial step with zero length, that is, the objective function value will be maintained, and a point corresponding to one of those dual variables that are $0$ will be picked to be removed from the support set. This is done as many times as necessary.

\subsubsection{$\Sbjpb$ is affinely dependent}

\noindent When $\Sbjpb$ is affinely dependent, we have seen that conditions (\ref{eq:curveCones},\ref{eq:curveAff}) do not define a curve. That means, a movement in the primal space is not possible. However, it is possible to move in the dual space to different dual solutions corresponding to the same $x^j$. Once one the dual variables becomes $0$, we drop that point from $\Ss^j$, fixing the issue of the affine dependence. In order to do that, we will solve the following problem

\begin{equation} \begin{array}{rl}\label{prob:affdep}
\max & \alpha^*\\
\st & \alpha_i(\alpha^*)\geq 0,\; i=1,...,s\\
%&x_0(\alpha^*)\leq x_0^j\\
&\alpha^*\geq 0
\end{array}\end{equation}

\noindent for $\alpha_i(\alpha^*)$ defined as in (\ref{eq:alphasAD}). This problem finds an affine combination of $x^j$ where $\alpha^*>0$ and $\alpha_i = 0$ for some $i=1,...,s$. Problem (\ref{prob:affdep}) can easily be solved using a minimum ratio rule, see Algorithm \ref{alg:affDep}. 

%From (\ref{eq:alphaAD}), it is easy to see that 
%\[\alpha_{1:s} = \rho+\alpha^*\sigma ,\]
%\noindent with
%\begin{equation}\label{eq:rs}
%\rho = \left[\begin{array}{c}
%1-1^T(u+x_0^jv)\\u+x_0^jv
%\end{array} \right]\myand \quad \sigma = \left[\begin{array}{c}
%-1-1^Tw\\w
%\end{array} \right].
%\end{equation}
%\noindent And so the solution to (\ref{prob:affdep}) is 
%\begin{equation}\label{eq:rs2}
%-\frac{\rho_k}{\sigma_k}:=\min_{j=1,\dots,s} \left\{-\frac{\rho_j}{\sigma _j}: \sigma_j<0\right\}.
%\end{equation}
%
%\smallskip

\begin{algorithm}
	\caption{Aff Dep Case procedure}\label{alg:affDep}
	\begin{algorithmic}[1]
		\Require Vectors $u,v,w$, point $x^j$.
		\Ensure Index $k$.		
		\smallskip
		\State Define vectors $\rho$ and $\sigma$:
		\begin{equation*}\label{eq:rs}
		\rho = \left(\begin{array}{c}
		1-1^T(u+x_0^jv)\\u+x_0^jv
		\end{array} \right)\myand \;\, \sigma = \left(\begin{array}{c}
		-1-1^Tw\\w
		\end{array} \right).
		\end{equation*}
		
		\State Get $k$ such that %$-\frac{\rho_k}{\sigma _k} = \min_{j=1,\dots,s} \left\{-\frac{\rho_j}{\sigma _j}: \sigma_j<0\right\}.$
		\begin{equation}\label{eq:rs2}
		-\frac{\rho_k}{\sigma_k}:=\min_{j=1,\dots,s} \left\{-\frac{\rho_j}{\sigma _j}: \sigma_j<0\right\}.
		\end{equation}
	\end{algorithmic}
\end{algorithm}

Let $p_k\in\Ss^j$ be the point resulting from (\ref{eq:rs2}). In section \ref{subsec:correctness} it will be proved that $\Sb^j\setminus\{\pb_k\}\cup\{\pb^*\}$ is affinely independent, and so, a  curve search can then be performed.

%\smallskip

\medskip

\subsubsection{The special case when $p^*$ is the solution and the relation to the $|\Ss^j|=1$ case}
\noindent Note that the curve search always assumes that the solution to $\InfQSjp$ will satisfy (\ref{eq:curveCones}) for a subset of $\Ss^j$. But that is not the case when $p^*$ happens to be the solution. In fact, as we show next, this case fails to be identified by the curve search procedure.

Consider that $p^*$ is the solution to $\InfQSjp$. After a series of partial steps where a point from $\Ss^j$ is dropped each time, the algorithm finally performs a curve search with $|\Ss^j|=1$, in which, the full step essentially consists on solving the following system of equations
\begin{subequations}\label{eq:aux10}
	\begin{align}
	\norm{\pb_{j_1}-\xb} = p_{j_10}-x_0\label{eq:aux10a}\\
	\norm{\pb^*-\xb} = p_0^*-x_0\label{eq:aux10b}\\
	\xb\in\aff(\{\pb_{j_1},\pb^*\})\label{eq:aux10c}.
	\end{align}
\end{subequations}
\noindent When $p^*$ is the solution to $Inf_\Q(\{p_{j_1}\}\cup\{p^*\})$, the solution to equations (\ref{eq:aux10}) does not exist in general (unless it happens that $\norm{\pb_1-\pb^*} = p_{10}-p^*_0$), and so from the full step we would have $x_0^{full}=-\infty$. On the other hand, the partial step consists of solving $\alpha_1(x_0)=0$ given simply by 
%\[\alpha^*(x_0) = \frac{|p_{{j_1}0}-x_0|}{\norm{\pb^*-\pb_{j_1}}} = \frac{p_{{j_1}0}-x_0}{\norm{\pb^*-\pb_{j_1}}},\]
\[\alpha_1(x_0) = 1 - \frac{p_{{j_1}0}-x_0}{\norm{\pb^*-\pb_{j_1}}}= 0,\]
\noindent (note that $M=[\,\,]$) for $x_0 \leq p_{{j_1}0}$. This yields
\begin{equation}\label{eq:auxS1}
x_0 = p_{10} - \norm{\pb^*-\pb_{j_1}}\quad\text{and}\quad \xb = \pb^*.
\end{equation}
\noindent  The partial step in this case does not get the correct value of $x_0$, which would be $p^*_0$. That happens because, in the curve search we made the assumption that $\norm{\pb_1-\xb}=p_{10}-x_0$. Note that the result of the partial step is independent of whether $p^*$ is the solution or not.

So, when $p^*$ is the solution to $Inf_\Q(\{p_{j_1}\}\cup\{p^*\})$, following the partial and full step procedures as explained before, we would have $x_0^{partial}<x_0^{full}$. As a consequence, a partial step would be taken, point $p_{j_10}$ would correctly be dropped from $\Ss^j$, but the next iterate would have the incorrect value of $x_0^j$!

This issue conducted us to add an extra step after the Optimality Check  where we check whether $p^*$ is the solution, that is, 
\[\norm{\pbij-\pb^*} \leq \pijo-p^*_0\quad \forall\,\pij\in\Ss^j.\]
\noindent If $p^*$ is not the solution, then a curve search is performed. This not only avoids issues with the partial step when $\Ss^j$ has a single point, but it also avoids potentially $n$ curve search procedures where a point would be removed from the support set in each one, only then to find out that $p^*$ is the solution.

\smallskip

This extra step does not fix the fact that, when $\Ss^j$ has a single point, the partial step will not work correctly. However that will never be an issue once it is known that $p^*$ is not the solution, because when that is the case, at the full step we will obtain 
\[x_0^{full} = \frac 12\left({p_{j_10}+ p^*_{0} -\norm{\pb_{j_1}-\pb^*}}\right)\]
\noindent as per Theorem \ref{theo:2pts}. The value of $x_0^{partial}$ will be as in (\ref{eq:auxS1}), and so $x_0^{full}>x_0^{partial}$ as a consequence of the fact that 
\[\norm{\pb_{1}-\pb^*} > p_{10}-p^*_0.\]

\smallskip

Alternatively, and for the sake of simplicity, whenever $|\Ss^j|=1$ one can simply use the formulas given by Theorem \ref{theo:2pts} instead of doing a curve search.

\subsection{Pseudo-code}\label{subsec:pseudo}

We now aggregate the results/discussion of the previous sections on Algorithm \ref{alg}.

\begin{algorithm}
	\caption{Dual algorithm for the infimum of $\Pp$ with respect to~$\Q$}\label{alg}
	\begin{algorithmic}[1]
		\Require $\Pp$, dual feasible S-pair $(\Ss^0, x^0)$.
		\Ensure $x$, $\Ss$, the optimal solution and an support set, respectively.	
		
		%\noindent \hspace{-15pt}\emph{Initialization:} 	
		%\State Choose any point $p\in\Pp$; $x^0\gets p$;  $\Ss^0\gets \{p\}$.
		\smallskip
		
		\For {$j=0,1,....$}
		
		\smallskip\hspace{-15pt}\emph{Optimality check:} 
		\If {$x^j\leqQ p_i$ for all $p_i\in\Pp$}
		\State $x^j$ is the optimal solution and $\Ss^j$ an optimal support set. \textbf{Stop.}
		\Else 
		\State Get $p^*\in\Pp$ s.t. $x^j\gQ p^*$.	
		\EndIf	
		
		%		\State If $x^j\leqQ p_i$ for all $p_i\in\Pp$, then $x$ is the optimal solution. \textbf{Stop.} 
		%		\State Else pick $p^*\in\Pp$ s.t. $x^j\gQ p^*$.		
		
		\medskip\hspace{-15pt}\emph{Check if solution of $\InfQSjp$ is $p^*$:}	
		\If {$p^*\leqQ \pij$ for all $\pij\in\Ss^j$}
		\State $x^{j+1}\gets p^*$; $\Ss^{j+1}\gets\{p^*\}$; Go to \emph{Optimality check}.
		\EndIf
		
		\medskip\hspace{-15pt}\emph{Special case $|\Ss^j|=1$ (optional):}	
		\If {$|\Ss^j|=1$}  
		\State Set $x^{j+1}$ as in Theorem \ref{theo:2pts}; $\Ss^{j+1}\gets\Ss^{j}\cup\{p^*\}$; Go to \emph{Optimality check}.
		\EndIf
		
		\medskip\hspace{-15pt}\emph{$\Sb^j\cup\{\pb^*\}$ affinely dependent:}	
		\If {$\norm{z} = 0$}  
		\State Get $k$ from Algorithm \ref{alg:affDep}; $\Ss^{j}\gets \Ss^{j}\setminus\{p_k\}$.
		\EndIf
		
		\medskip	\hspace{-15pt}\emph{Curve search:} 	
		\State Define matrix $M$ and vectors $b,c,u,v,w$ and $z$.
		\State Get $x_0^{partial}$ from Algorithm \ref{alg:partial}; Get $x_0^{full}$ from Algorithm \ref{alg:full}.
		
		\smallskip
		
		\If {$x_0^{partial}\geq x_0^{full}$}
		\State $x_0^{j}\gets x^{partial}_0$; $\xb^{j} \gets \Gamma^+( x_0^{j});$ $\Ss^{j}\gets \Ss^{j}\setminus\{p_k\}$; Go to \emph{Curve search}.
		\Else
		\State $x_0^{j+1}\gets x^{full}_0$;	$\xb^{j+1} \gets \Gamma^+(x_0^{j+1})$; $\Ss^{j+1}\gets\Ss^{j}\cup\{p^*\}$; Go to \emph{Optimality check}.
		\EndIf
		\EndFor
	\end{algorithmic}
\end{algorithm}

\newpage 

\subsection{Finiteness and correctness of the algorithm} \label{subsec:correctness}

For the proof of the correctness of the algorithm presented in Algorithm \ref{alg}, we will follow a similar approach as in \cite{Goldfarb83}. We first introduce the following definition:

\begin{definition}\label{def:Vtriple}
	The triple $(x, \Ss, p)$ is said to be a \emph{(violated) V-triple} if the following four conditions hold:
	\begin{enumerate}[(a)]
		\item $\Sb\cup\{\pb\}$ is affinely independent,
		\item $\norm{\pb-\xb}\geq(p_0-x_0)$,
		\item $\norm{\pb_i-\xb}=(p_{i0}-x_0)$, for $p_i\in\Ss$,
		\item $\xb\in\conv(\Sb\cup\{\pb\})$.
	\end{enumerate}	
\end{definition}

We now prove a series of theorems that will culminate in the correctness of the algorithm. 

\begin{theorem}\label{theo:corrFull}
	Given a V-triple $(x^j, \Ss^j, p^*)$, if the solution given by the full step, $x^{full}$, exists and is dual feasible, then it is the optimal solution to $\InfQ(\Ss^j\cup \{p^*\})$. Moreover, $x_0^{full}\leq x_0^j$, and, if $x_0^{full}>x_0^{partial}$ then $(\Ss^j\cup \{p^*\}, x^{full})$ is a dual feasible S-pair.
\end{theorem}
\begin{proof}The statement is a direct consequence of the fact that the full step finds a point $x^{full}$ such that 
	\begin{align}
	\norm{\pb^*-\xb^{full}} = p_0^*-x_0^{full}&\\
	\norm{\pb_i-\xb^{full}}= p_{i0} -x_0^{full} &\quad \forall\, p_i\in\Ss,\\
	\xb^{full}\in\affSp.&
	\end{align}
	\noindent whenever the above conditions are feasible. If, additionally, 	$\xb^{full}\in\convSjp$ then $x^{full}$ solves $\InfQ(\Ss^j\cup \{p^*\})$ as per Corollary \ref{theo:convhull}. The fact that $x_0^{full}\leq x_0^j$ is a consequence of the full step definition.

	Now assume $x_0^{full}>x_0^{partial}$. It is easy to see that $(\Ss^j\cup \{p^*\}, x^{full})$ is a dual feasible S-pair. Clearly $\Sb^j\cup\{\pb^*\}$ is affinely independent. Additionally, $\xb^{full}\in\ri\conv(\Sb^j)$, because otherwise there would have been a coefficient of the convex combination that would be $0$, implying that $x_0^{full}=x_0^{partial}$, which contradicts the assumption.
\end{proof}

\noindent From Theorem \ref{theo:corrFull} we conclude that, before the ``Optimality Check'' procedure, we always have a dual feasible S-pair $(\Ss^j,x^j)$. After picking a $p^*$ corresponding to an infeasible constraint, if $\Sb^j\cup\{\pb^*\}$ is affinely independent, then we have a V-triple. Otherwise, Theorem \ref{theo:corrAffDep}, proves that Algorithm \ref{alg:affDep} returns a V-triple. Either way, the V-triple at this stage is always such that $\norm{\pb^*-\xb^j}>(p_0^*-x_0^j)$. Combined with Theorem \ref{theo:corrPart}, we conclude that before a curve search we always have a V-triple.

\begin{theorem}\label{theo:corrPart}
	Suppose $x_0^{partial}\geq x_0^{full}$. Given a V-triple $(x^j, \Ss^j, p^*)$, the partial step returns another V-triple $({x}^{partial}, \Ss^j\setminus\{p_k\}, p^*)$ such that ${x}_0^{partial}\leq x_0^j$.

\end{theorem}	
\begin{proof}We need to prove properties (a-d) from Definition \ref{def:Vtriple} for $({x}^{partial}, \Ss^j\setminus\{p_k\}, p)$. Property (a) follows from the fact that $\Sb^j\cup\{\pb^*\}$ is affinely independent. Properties (c) and (d) and the fact that ${x}_0^{partial}\leq x_0^j$ are a direct consequence of the partial step definition. Property (b) follows from a continuity argument since $\norm{\pb^*-\xb^j}\geq(p_0^*-x_0^j)$ and the fact that $x_0^{partial}\geq x_0^{full}$, implying that $\norm{\pb^*-\xb^{partial}}\geq(p_0^*-x_0^{partial})$. 
\end{proof}

\begin{theorem}\label{theo:corrAffDep}
	When $\Sb^j\cup\{\pb^*\}$ is affinely dependent, $(x^j, \Ss^j\setminus\{p_k\}, p^*)$ for $p_k$ satisfying (\ref{eq:rs2}) is a V-triple.
\end{theorem}
\begin{proof}In order to prove that $(x^j, \Ss^j\setminus\{p_k\}, p^*)$ we need to prove the four properties from Definition \ref{def:Vtriple}. Properties (b) and (c) are trivial since $(\Sb^j, x^j)$ is an S-pair and $p^*$ corresponds to a primal infeasible constraint at $x^j$. Now, in order to prove (a), suppose, by contradiction, that $\Sb^j\setminus\{\pb_k\}\cup\{\pb^*\}$ is affinely dependent. Then, since $\Sb^j\setminus\{\pb_k\}$ is affinely independent, there exists $\beta_i$, $i=1,...,s$ and $i\neq k$ s.t.
	\[\pb^* = \sum_{i=1,\, i\neq k}^{s}{\beta_i\pb_i}\quad \text{with}\sum_{i=1,\,i\neq k}^{s}\beta_i=1.\]
	\noindent With loss of generality consider $k\neq 1$. Since $z=Mw+(\pb-\pb_1)=0$ (a consequence of the affine dependence of $\Sb^j\cup\{\pb^*\}$), we have 
	\[\pb^* =  \sum_{i=2}^s w_i(\pb_i-\pb_1)+\pb_1 = \sum_{i=1}^{s}{\delta_i\pb_i}  \]
	\noindent for $\delta_1=1-\sum_i w_i$ and $\delta_i=w_i$ for $i=2,...,s$. As a consequence 
	\[\pb_k =  \sum_{i=1,\, i\neq k}^{s}{\frac{\beta_i-\delta_i}{\delta_k}\pb_i}.\]
	\noindent We have that $\sum \frac{\beta_i-\delta_i}{\delta_k} = 1$, and note that $\delta_k\neq 0$. That is, $\pb_k$ is an affine combination of $\Sb\setminus\{\pb_k\}$, which contradicts the assumption.
	
	Finally, consider $\beta=\rho-\frac{\rho_k}{\sigma_k}\sigma$ and $\beta^*=\frac{\rho_k}{\sigma_k}$. It is easy to see that $\beta\geq 0$ and, in particular, $\beta_k=0$. We now prove that $\xb^j=M\beta_{2:s}+\beta^*(\pb-\pb_1)+\pb_1$, that is property (d), by observing that
	\begin{align}
	M\beta_{2:s}+\beta^*(\pb-\pb_1)+\pb_1 &= M\left(u+x_0v-\frac{\rho_k}{\sigma_k}w\right) - \frac{\rho_k}{\sigma_k}(\pb-\pb_1)+\pb_1\\
	&=M(u+x_0v) + \pb_1 - \frac{\rho_k}{\sigma_k}\left(Mw+(\pb-\pb_1)\right)\\
	&=M(u+x_0v) + \pb_1\\
	&=\xb^j
	\end{align}
	\noindent since $z=0$. \end{proof}

\medskip 

With the previous theorems, we conclude that, starting from a V-triple $(x^j,\Ss^j,p^*)$ for which $\norm{\pb^*-\xb^j}>(p_0^*-x_0^j)$, one can obtain a dual feasible S-pair $(\hat{\Ss}\cup\{p\},\hat{x})$ such that $\hat{\Ss}\subseteq\Ss^j$ in at most $|\Ss^j|-|\hat{\Ss}|\leq n$ partial steps and a full step. Moreover, $\hat{x}_0<x_0^j$, since, even though when taking a partial step or a full step the value of $x_0$ may be maintained, we know that the value of $x_0$ must decrease either in one of the partial steps taken or in the full step, because otherwise we would have that $\hat{x}=x^j$, that is, $Inf_\Q(\Ss^j)=Inf_\Q(\hat{\Ss}\cup\{p^*\})$ contradicting the fact that $p^*$ is infeasible to $Inf_\Q(\Ss^j)$. Therefore, since the value of $x_0$ strictly decreases at each major iteration the same S-pair can never reoccur. Since the number of possible S-pairs is finite we conclude:

\begin{theorem}
	The proposed dual algorithm solves problem $Inf_\Q(\Pp)$ in a finite number of iterations.
\end{theorem}

\subsubsubsection{An observation on degeneracy}

\noindent It is easy to see that in a primal algorithm \emph{cycling} would be a possibility. The current primal feasible solution may correspond to different support sets, that is, different dual solutions. Thus, it could happen that, after a sequence of adding/dropping points from the support set, the algorithm did not move in the primal space and ended up in a support set visited previously. That is a consequence of the lack of freedom of which point to enter the support set in the primal setting, it has to be one corresponding to a primal constraint that is active. In the dual algorithm, the equivalent situation, of when a movement is not possible because a dual variable is zero at the current iterate, is easily dealt by the algorithm by removing the point corresponding to that dual variable from $\Ss^j$ (either at a partial step or when $\Sb^j\cup\{\pb^*\}$ is affinely dependent). This is done as many times as the number of dual variables that are zero, after which a movement will then be possible in the next iteration. Therefore, as far as degeneracy is concerned no special procedure is required in our algorithm to prevent \emph{cycling}.

\section{Implementation details} 
The main  computational work that is required in the algorithm happens before each curve search, when three linear systems need solved in order to get vectors $u, v$ and $w$. These linear systems all have the same matrix $M^TM$. Our implementation is based on the QR factorization of matrix $M$ of size $n\times (s-1)$ such that $s-1\leq n$, where $s=|\Ss^j|$. We have that $M=QR$, with $Q$ a $n\times n$ orthogonal matrix and $R$ a $n\times (s-1)$ upper triangular matrix. Let 
\begin{equation*}
Q=\left[Q_M\;|\;Q_\perp\right]\myand R= \left[\begin{array}{c} R_\triangle \\\hline 0  \end{array}\right],
\end{equation*}
\noindent with $Q_M$ with size $n\times (s-1)$ and $Q_\perp$ with size $n\times (n-s+1)$, and let, and $R_\triangle$ a $(s-1)\times (s-1)$ matrix. As a consequence, $M^TM = R^TR$. 

Using the QR factorization of $M$ allows the following linear systems
\[(M^TM)u = b-M^T\pb_{j_1}\myand (M^TM)v = c, \]
\noindent to be reduced to two triangular linear systems, which can efficiently be computed using \emph{Back Substitution}. Moreover, $M^+=R_\triangle^{-1}Q_M^T$, and so \[w=-M^+(\pb_{1}^*-\pb_{j_1})\]
\noindent can be obtained by solving the triangular linear system
\[R_\triangle w=-Q_M^T(\pb_{1}^*-\pb_{j_1}).\]

Matrix $M$ is given by the points of the support set $\Ss^j$ which is updated throughout the algorithm either by adding a point or removing a point. Next we explain how to update $M$ in those circumstances.

\smallskip
\noindent\textbf{Adding a point to $\Ss^j$.} Whenever a point, $p_{i_k}$ is added to $\Ss^j$, we append column $\pb_{i_k}-\pb_{i_1}$ to matrix $M$.
\smallskip

\noindent \textbf{Removing a point from $\Ss^j$.} Whenever a point, $p_{i_k}$ is removed from $\Ss^j$, there are two cases. If $i_k\geq 1$ then we need to remove the $(k-1)$-th column from $M$. When $i_k=1$, we need to remove the first column from $M$, obtaining $\hat M$, and get the new $M$ by adding the rank-$1$ matrix $ (p_{i_1}-p_{i_k})^T1_n$ to $\hat M$.
\smallskip

Given the above, every time $M$ needs to be updated, we can use the QR factorization of the old matrix $M$ to calculate efficiently the QR factorization of the new $M$ \cite[\S 12.5]{Golub96}. This is accomplished by using \emph{Givens rotations}, which, in the case when $M$ is $n\times n$, has $\bigO(n^2)$ computational complexity.

\label{sec:implementation}

\section{Computational Results}\label{sec:results} We have implemented our algorithm in MATLAB as explained in Section \ref{sec:implementation} choosing the most infeasible point at each iteration. We also compared it with solving both the primal (\ref{primal}) and dual (\ref{dual}) problems with the interior point method solver of Gurobi \cite{gurobi} (version 8.0.0) using its MATLAB interface. All Gurobi parameters were kept at their default values. Our experiments were conducted using MATLAB R2016a (version 9.0) on a Mac with an Intel Core i5 1.6 GHz processor, with 8GB RAM, running Mac OS~X version 10.11.6. The results are shown in Table \ref{tab:results}.

\begin{table}[htbp]
	\centering
	\begin{tabular}{|cc|ccc|cc|cc|}
		\multicolumn{2}{|c|}{\textbf{Problems}} & \multicolumn{3}{c|}{\textbf{Our dual algorithm}} & 			
		\multicolumn{2}{c|}{\textbf{Gurobi (dual)}} & \multicolumn{2}{c|}{\textbf{Gurobi (primal)}}\\
 &&  & \multicolumn{1}{c}{S-pair} & Time &  &Time &  &Time \\
		\multicolumn{1}{|c}{$n$} & \multicolumn{1}{c|}{$m$} & \multicolumn{1}{c}{Iters} & \multicolumn{1}{c}{updates} & \multicolumn{1}{c|}{(secs)} & \multicolumn{1}{c}{Iters} & \multicolumn{1}{c|}{(secs)} & \multicolumn{1}{c}{Iters} & \multicolumn{1}{c|}{(secs)} \\
		\hline
		\hline
		
		$10 $   & $10^2$   & 6.44     & 6.68     & 0.008 & \multicolumn{1}{r}{9.72} & \multicolumn{1}{r|}{0.008}& \multicolumn{1}{r}{8.56} & \multicolumn{1}{r|}{0.228} \\
		$10 $   & $10^3$ & 8.12     & 8.76     & 0.010 & \multicolumn{1}{r}{11.56} & \multicolumn{1}{r|}{0.094} & \multicolumn{1}{r}{9.16} & \multicolumn{1}{r|}{1.627} \\
		$10$    & $10^4$ & 8.12     & 9.24     & 0.027 & \multicolumn{1}{r}{12.44} & \multicolumn{1}{r|}{1.073} & \multicolumn{2}{c|}{out of memory} \\
		\hline
		$10^2$   & $10^2$   & 15.56    & 15.56    & 0.031 & \multicolumn{1}{r}{9.56} & \multicolumn{1}{r|}{0.127} & \multicolumn{1}{r}{8.88} & \multicolumn{1}{r|}{7.455}\\
		$10^2$   & $10^3$  & 20.96    & 21.16    & 0.145 & \multicolumn{2}{c|}{could not solve$^\dag$} & \multicolumn{1}{r}{9.36} & \multicolumn{1}{r|}{87.601}\\
		$10^2$  & $10^4$& 25.28    & 25.72    & 0.375 & \multicolumn{2}{c|}{could not solve$^\dag$} &    \multicolumn{2}{c|}{out of memory}      \\
		$10^2$   & $10^5$ & 29.24    & 29.88    & 3.341 & \multicolumn{2}{c|}{could not solve$^\dag$} &         \multicolumn{2}{c|}{out of memory}       \\
		\hline 			
		$10^3$  & $10^4$ & 77.32    & 77.44    & 8.641 & \multicolumn{2}{c|}{could not solve$^\dag$} &        \multicolumn{2}{c|}{out of memory}       \\
		$10^3$ & $10^5$ & 87.92    & 88.16    & 154.613 & \multicolumn{2}{c|}{could not solve$^\dag$} &          \multicolumn{2}{c|}{out of memory}      \\
	\end{tabular}
	\caption{Computational results with the averages corresponding to 25 datasets with points randomly sampled from a standard normal distribution. $^\dag$Matlab stopped responding.}\label{tab:results}
\end{table}

One reason for the good performance of the dual algorithm is that it does not add many points to the support set that are not in the final support set. This can be seen by the fact that the number of iterations (number of points added to the support set) is very close to the number of dual S-pair updates (number of curve searches performed).

\newpage

%############################# REFERENCES #############################
\bibliography{\bibA,\bibB,\bibC,\bibD,\bibE}%no spaces after the commas!
\bibliographystyle{siam}
\end{document}